\documentclass[12pt,reqno]{amsart}
\usepackage[a4paper, hmargin={2.7cm,2.7cm},vmargin={3.3cm,3.3cm}]{geometry}
\usepackage[english]{babel}
\usepackage{cite}
\usepackage{amsmath,amssymb, amsfonts, amsthm}
\usepackage{enumerate}
\usepackage[OT2,T1]{fontenc}
\usepackage{dsfont}
\usepackage{etoolbox}

\makeatletter
\patchcmd{\ttlh@hang}{\parindent\z@}{\parindent\z@\leavevmode}{}{}
\patchcmd{\ttlh@hang}{\noindent}{}{}{}
\makeatother
\usepackage{color}

\newcommand\numberthis{\addtocounter{equation}{1}\tag{\theequation}}

\newtheorem{theorem}{Theorem}[section]
\newtheorem{lemma}[theorem]{Lemma}
\newtheorem{proposition}[theorem]{Proposition}
\newtheorem{corollary}[theorem]{Corollary}

\theoremstyle{definition}
\newtheorem{definition}[theorem]{Definition}

\theoremstyle{remark}
\newtheorem{remark}[theorem]{Remark}

\newcommand{\SImodZ}{SI/Z}
\newcommand{\Hpi}{\mathcal{H}_{\pi}}
\newcommand{\Hsmooth}{\Hpi^{\infty}}
\newcommand{\Hone}{\mathcal{H}^1_{\pi}}
\newcommand{\Honedual}{(\Hone)^{\urcorner}}
\newcommand{\Api}{\mathcal{A}_{\pi}}
\newcommand{\Bpi}{\mathcal{B}_{\pi}}
\newcommand{\xdot}{\dot{x}}

\newcommand{\Co}{Co}
\newcommand{\CoLp}{\Co(L^p(G))}
\newcommand{\CoLone}{\Co(L^1(G))}
\newcommand{\CoLinfty}{\Co(L^{\infty} (G))}
\newcommand{\Conull}{\Co (C_0 (G))}
\newcommand{\WLst}{W^{\st}(L^{\infty}, L^1)}
\newcommand{\WRst}{W^{\st}(L^{\infty}, L^1)}
\newcommand{\WRstw}{W^{\st}(L^{\infty}, L^1_{w_{\alpha}})}
\newcommand{\WL}{W(L^{\infty}, L^1)}

\newcommand{\WR}{W_R(L^{\infty}, L^1)}
\newcommand{\WLC}{W(C_0, L^1)}
\newcommand{\WRC}{W_R(C_0, L^1)}
\newcommand{\WLM}{W(M, L^{\infty})}
\newcommand{\MG}{M(G)}
\newcommand{\st}{st}
\newcommand{\CS}{\pi(\Lambda) g}
\newcommand{\psike}{\psi_k^{\varepsilon}}
\newcommand{\psije}{\psi_j^{\varepsilon}}
\newcommand{\schur}{Schur}

\DeclareMathOperator*{\ran}{ran}
\DeclareMathOperator*{\trace}{trace}

\DeclareMathOperator*{\loc}{loc}

\DeclareMathOperator*{\sep}{sep}
\DeclareMathOperator*{\rel}{rel}
\DeclareMathOperator*{\spann}{span}
\DeclareMathOperator*{\supp}{supp}

\DeclareMathOperator*{\esssup}{ess\,sup}

\DeclareMathOperator*{\id}{id}

\DeclareSymbolFont{cyrletters}{OT2}{wncyr}{m}{n}
\DeclareMathSymbol{\Sha}{\mathalpha}{cyrletters}{"58}

\title{Balian-Low type theorems on homogeneous groups}
\author[K. Gr\"{o}chenig]{Karlheinz Gr\"{o}chenig}
\address{Faculty of Mathematics, 
University of Vienna, 
Oskar-Morgenstern-Platz 1, 
A-1090 Vienna, Austria}
\email{karlheinz.groechenig@univie.ac.at}

\author[J.L. Romero]{Jos\'e Luis Romero}
\address{Faculty of Mathematics,
University of Vienna,
Oskar-Morgenstern-Platz 1,
A-1090 Vienna, Austria\\and
Acoustics Research Institute, Austrian Academy of Sciences,
Wohllebengasse 12-14 A-1040, Vienna, Austria}
\email{jose.luis.romero@univie.ac.at, jlromero@kfs.oeaw.ac.at}

\author[D. Rottensteiner]{David Rottensteiner}
\address{Department of Mathematics: Analysis, Logic and Discrete Mathematics,
Ghent University, Belgium}
\email{david.rottensteiner@ugent.be}

\author[J.T. van Velthoven]{Jordy Timo van Velthoven}
\address{Faculty of Mathematics, 
University of Vienna, 
Oskar-Morgenstern-Platz 1, 
A-1090 Vienna, Austria}
\email{jordy-timo.van-velthoven@univie.ac.at}

\subjclass[2010]{22E25, 22E27,42C15, 42C40}
\date{}
\keywords{Balian-Low type theorem, deformation theory, homogeneous group, 
localized frame, off-diagonal decay, spectral invariance, strict density condition}
 \thanks{K.\ G.\ was
   supported in part by the  project P31887-N32  of the
 Austrian Science Fund (FWF). 
 D.\ R. was supported by the Austrian Science Fund (FWF) project I 3403.
 J.~L.~R.~gratefully acknowledges support from the Austrian Science Fund (FWF): P 29462 - N35,
from the WWTF grant INSIGHT (MA16-053).
J.~v.~V.~acknowledges support from the Austrian Science Fund (FWF): P 29462 - N35.}

\begin{document}

\maketitle

\begin{abstract}
We prove strict necessary density conditions 
for coherent frames and Riesz sequences on homogeneous
groups. Let $N$ be a connected, simply connected nilpotent Lie group 
with a dilation structure (a homogeneous
group) and let $(\pi, \Hpi)$ be an irreducible, square-integrable representation modulo the
center $Z(N)$ of $N$ on a Hilbert space $\Hpi$ of formal dimension
$d_\pi $. If $g \in \Hpi$ is an integrable vector and the  set $\{ \pi
(\lambda )g : \lambda \in \Lambda \}$ for a discrete subset $\Lambda
\subseteq N / Z(N)$ forms a frame for $\Hpi$, then its density
satisfies the strict inequality  $D^-(\Lambda )> d_\pi $, where
$D^-(\Lambda )$ is the  lower Beurling
density. An analogous density condition $D^+(\Lambda) < d_{\pi}$ holds for a
Riesz sequence in $\Hpi$ contained in the orbit of $(\pi, \Hpi)$.
The proof is based on a deformation theorem for coherent systems, a
universality result for $p$-frames and $p$-Riesz sequences, some results
from Banach space theory, and tools from the analysis on homogeneous groups.  
\end{abstract}

\section{Introduction}
Let $G$ be a compactly generated, locally compact group of polynomial growth, 
and let $(\pi, \Hpi)$ be an irreducible, square-integrable representation of $G$ 
on a Hilbert space $\Hpi$ of formal dimension $d_{\pi}$. 
We consider the spanning properties of discrete systems in the orbit of 
a vector $g \in \Hpi$ under $(\pi, \Hpi)$, 
\begin{align} \label{eq:intro_coherent}
\CS = \big\{\pi(\lambda)g: \lambda \in \Lambda \big\}, 
\end{align}
and study  their relation with the density of the index set $\Lambda \subseteq G$.
The existence of frames and Riesz sequences of the form \eqref{eq:intro_coherent} 
for ``sufficiently dense'' respectively ``sufficiently sparse'' 
index sets is well-known
\cite{feichtinger1989banach1, fuehr2007sampling, groechenig1991describing}.
Necessary density conditions follow from the abstract theory of
localized frames \cite{MR2224392, MR2235170, MR1325536}, 
or sampling in reproducing kernel Hilbert spaces
\cite{ fuehr2017density, mitkovski1}. 
For example, \cite[Theorem 5.3]{fuehr2017density}
asserts that:
\begin{align}
&\text{If $\CS$ forms a frame for $\Hpi$, then $D^- (\Lambda) \geq d_{\pi}$.} \numberthis \label{intro_frame}
\\
&\text{If $\CS$ forms a Riesz sequence in $\Hpi$, then $D^+ (\Lambda) \leq d_{\pi}$.} 
\numberthis \label{intro_riesz}
\end{align}
The quantities $D^- (\Lambda)$ and $D^+ (\Lambda)$ denote certain
lower respectively upper Beurling densities of $\Lambda$; see
\cite[Section 5.3]{fuehr2017density} for the precise details.

In this paper we consider the question of the strictness of the density conditions \eqref{intro_frame} and \eqref{intro_riesz}.  
This problem has been studied extensively in the setting of the Heisenberg group 
and is generally known as the Balian-Low theorem in Gabor theory. 
Precisely, the Balian-Low theorem asserts that if $g \in L^2 (\mathbb{R})$ is a phase-space localized function, 
then the family of functions $\{ e^{2\pi i l \cdot } g(\cdot + k) \; : \; k, l \in \mathbb{Z} \}$ 
does not form a Riesz basis or a frame for $L^2 (\mathbb{R})$; see \cite{daubechies90, battle1988heisenberg}. 
In view of recent constructions of orthonormal bases in the orbit of unitary representations of 
nilpotent Lie groups \cite{MR3864505} 
and solvable semi-direct products \cite{oussa2, MR3743195}, 
it is of interest whether a Balian-Low type theorem holds 
for other groups than the Heisenberg group. 
In this regard we recall the  famous Kirillov lemma  asserting that any nilpotent Lie group 
admits a subgroup isomorphic to the Heisenberg group. 
 Hence it is expected that a Balian-Low type theorem holds 
for general nilpotent Lie groups.

While the Balian-Low theorem is usually studied as a no-go result, it
also has a  fruitful interpretation  as a strict necessary density condition. The index set
of the system $\{ e^{2\pi i l \cdot } g(\cdot + k) \; : \; k, l \in \mathbb{Z}
\}$ is $\mathbb{Z}^2$ and  possesses the critical density $1$. The Balian-Low
theorem now states that the index set of a frame consisting of time-frequency shifts
with a phase-space localized generator must necessarily have super-critical
density.

Our goal is to prove that for coherent systems of the form
\eqref{eq:intro_coherent} with an integrable vector $g \in \Hpi$
 the inequalities in
\eqref{intro_frame} and \eqref{intro_riesz}   must be strict. 
One successful approach to derive such strict density conditions is to
study the deformations of frames and Riesz sequences first and then proceed by
contradiction. Suppose that a coherent
system~\eqref{eq:intro_coherent}  is a frame, and that
$\Lambda \subseteq G$ attains the critical density. The challenge is
to produce a suitable  \emph{deformation} of $\Lambda$ that still
yields a frame, but has  smaller density, thus contradicting the
non-strict density conditions \eqref{intro_frame}. 
A similar argument can be used to contradict \eqref{intro_riesz}. 
This line of argument goes back to Beurling \cite{MR1057613}, and 
variations of this program have been implemented several times, 
e.g., in complex analysis in
\cite{MR1655834} and \cite{groechenig2018strict} and in Gabor analysis
in \cite{MR2031050,MR3192621} and \cite{groechenig2015deformation}.

In the setting of a non-Abelian group, we first  need to  find  an
adequate deformation. Here  the theory of Lie groups offers a natural
setting with  an obvious choice: these are groups  endowed with a
family of dilations compatible with the group structure, namely the
so-called homogeneous nilpotent Lie groups. Although these groups are
non-Abelian  and possess a rich representation theory that is rather
different from $\mathbb{R}^d$, their geometry and measure theory are 
quite similar in the sense that they form a space of homogeneous type 
and many real-variable methods carry over to homogeneous
groups~\cite{folland1982hardy}. These structural similarities allow us to prove
the following Balian-Low type theorem, phrased in terms of the (homogeneous)
lower and upper Beurling densities.  

\begin{theorem} \label{thm:balianlow_homogeneous}
Let $N$ be a homogeneous Lie group and let $(\pi, \Hpi)$ be an 
irreducible representation of $N$ that is square-integrable modulo its center $Z(N)$. 
Let $\Lambda \subset N/Z(N)$ be a discrete subset. 
Let $g \in \Hpi$ be such that $\int_{N/Z(N)} | \langle g, \pi(x) g \rangle | \; d\mu_{N / Z(N)} (\xdot) < \infty$. 
\begin{enumerate}[(i)]
\item If $\CS$ forms a frame for $\Hpi$, then $D^- (\Lambda) > d_{\pi}$.
\item If $\CS$ forms a Riesz sequence in $\Hpi$, then $D^+ (\Lambda) < d_{\pi}$. 
\end{enumerate}
\end{theorem}

\begin{corollary}
  If $\{\pi (\lambda ) g: \lambda \in \Lambda \}$ is an orthonormal
  basis or a Riesz basis for $\Hpi$, 
  then $\int_{N / Z(N)} | \langle g, \pi(x) g \rangle | \; d\mu_{N/Z(N)} (\xdot) = \infty$; 
  in particular, $g \in \Hpi$ 
   cannot be a smooth vector of $(\pi, \Hpi)$. 
\end{corollary}

For the proof of Theorem \ref{thm:balianlow_homogeneous}, 
we will revisit the theory of deformations \cite{groechenig2015deformation} 
and localizable reproducing kernel Hilbert spaces \cite{groechenig2018strict} and 
follow their outline. 
Once it is
understood how to move from $\mathbb{R}^d$ to a general homogeneous group, many
arguments carry over  almost word-by-word. However, several important  steps require
technical modifications. In particular,  we
will prove the existence of coherent frames whose canonical dual frame consists
of a set of molecules in the sense of~\cite{groechenig2009molecules}. 
Although this is  only an  auxilliary
result, our Proposition~\ref{prop:canonical_dual} may be of  independent
interest as it  enriches our knowledge of abstract coorbit theory. 

From a larger perspective one might aim at more  abstract Balian-Low type
theorems for certain classes of  frames. This program leads to several
interesting problems in frame theory that are still open. According to
our current understanding at least the following
ingredients are required:

(i) One needs an appropriate deformation theory for frames. As argued
in~\cite{groechenig2015deformation}, one must go beyond local jitter
errors and  must consider global deformations.

(ii) One needs a class of suitable deformations that change the
density of the frame. For simplicity here we 
 treat only dilations. We
remark, however, that the more refined nonlinear deformation theory
developed in \cite{groechenig2015deformation} can be proved for other types
of Lie groups  and is applicable in other
contexts, including the evolution of coherent systems through
Hamiltonian flows \cite{MR3500423}.
An interesting  open problem is to find an appropriate  notion of  deformation for
general nilpotent groups or groups of polynomial growth. 

(iii) One needs some off-diagonal decay of the Gramian matrix of the
frame. In the context of Theorem~\ref{thm:balianlow_homogeneous} this
is expressed by the condition that $g \in \Hpi $ is an integrable vector. 
The off-diagonal decay paves the way for
the application of Banach algebra techniques, in particular statements
about spectral invariance and the off-diagonal decay of the inverse
matrix. See~\cite{Gr10} for a survey.  Our version for homogeneous groups is contained in
Propositions~\ref{prop:sjostrand-wienertype} and
\ref{thm:wiener-type-extension} with  almost the same proof as
Sj\"ostrand's~\cite{sjostrand1995wiener}. As a consequence our results also hold for $p$-frames and the frame property is then independent of $p$. The spectral invariance is
usually hidden in the technical part, but in our opinion it is at the
heart of strict density conditions. This view is supported by the fact
that the  Paley-Wiener space admits an orthonormal basis at the
critical density (this is the Shannon-Whittaker-Kotelnikov sampling
theorem) when at the same time its kernel lacks sufficient
decay. Likewise there are orthonormal Gabor bases, but their generating
function is not phase-space localized.\footnote{We note
  that this picture cannot be completely accurate, as there are
  shift-invariant spaces with exponentially decaying generator that
  admit a frame of reproducing kernels at the critical density.}

(iv) Last but not least, one needs some translation structure so that the
underlying configuration space looks the same everywhere. This is
evident in the context of groups, but far from obvious in other
examples. In fact,  the construction of suitable
translation operators in general Fock spaces was one of the major
innovations in  \cite{MR1655834} for the proof of strict density
conditions.

The paper is organized as follows: In Section~2 we collect the
required facts about homogeneous groups and their representations and
the background on coorbit spaces. In Section~3 we study the stability
of coherent frames and Riesz sequences under dilations. The main theorems are
then stated and proved in Section~4. The appendix offers the necessary
tools about the stability and off-diagonal decay of matrices indexed
by discrete subsets of a homogeneous group. We prove a new result
about the canonical dual frame of a coherent frame. 

\section{Coorbit space theory} \label{sec:preliminaries}

\subsection{Homogenous groups}
This section consists of preliminary results on homogeneous groups 
and sets up the notation used throughout the paper. 
Standard references on homogeneous groups are 
\cite{folland1982hardy, fischer2016quantization}. 

 A family of \emph{dilations} $\{D^{\mathfrak{g}}_r \}_{r>0}$ of a Lie algebra $\mathfrak{g}$ 
is a family of Lie algebra automorphisms 
$D^{\mathfrak{g}}_r : \mathfrak{g} \to \mathfrak{g}$  
of the form
$
D^{\mathfrak{g}}_r = \exp_{\mathrm{GL}(\mathfrak{g})} \big(A \ln (r) \big)
$
for some diagonalizable linear operator $A : \mathfrak{g} \to \mathfrak{g}$
with positive eigenvalues, called the \emph{dilations' weights}. 
 A connected, simply connected Lie group $G$ is called \emph{homogeneous} if
 its Lie algebra $\mathfrak{g}$ is equipped with a family of dilations. 
 The \emph{homogeneous dimension} of $G$ is the number $Q := \trace(A)$. 
Throughout this paper, it will be assumed, without loss of generality, 
that the dilations' lowest weight equals $1$.  We have $Q \geq \dim(G)$. 

A Lie algebra admitting a family of dilations is nilpotent, and hence so is its associated
connected, simply connected Lie group. 
The converse does not hold, i.e., not every nilpotent Lie group is homogeneous \cite{dyer1970nilpotent},
although they exhaust a large class \cite{johnson1975homogeneous}. 

Any dilation $D^{\mathfrak{g}}_r$ on $\mathfrak{g}$ 
induces a continuous group automorphism $D_r : G \to G$
defined by
\[ D_r := \exp \circ D_r^{\mathfrak{g}} \circ \exp^{-1}, \]
where $\exp : \mathfrak{g} \to G$ is the exponential map. 

A \emph{homogeneous norm} on $G$ is a continuous mapping $| \cdot |_G : G \to [0,\infty)$ 
satisfying
\begin{enumerate}[(i)]
\item $|x^{-1} |_G = |x|_G$ for all $x \in G$;
\item $|D_r (x) |_G = r | x |_G$ for all $x \in G$ and $r>0$;
\item $|xy|_G \leq |x|_G + |y|_G$ for all $x,y\in G$;
\item $|x|_G = 0$ if, and only if, $x = e$. 
\end{enumerate}

Every homogeneous group admits a homogeneous norm, and 
the mapping
\[ d_G : G \times G \to \mathbb{R}^+, \quad (x,y) \mapsto |x^{-1} y|_G \]
forms a left-invariant metric on $G$, the so-called \emph{homogeneous metric}. 

The open ball $B_R (x)$ in $G$ of radius $R>0$ and center $x \in G$ is defined as
$
B_R (x) :=  \{ y \in G \; : \; |x^{-1} y |_G < R \}. 
$
For any $x, y \in G$ and $R>0$, we have
$y B_R (x) = B_R (yx)$ and $B_R (e) = D_R (B_1 (e))$. 
The Haar measure $\mu_G$ on $G$ satisfies
$\mu_G (D_r (E)) = r^Q \mu_G (E)$ for every Borel measurable set $E \subseteq G$. 
Any ball $B_R (x)$ is relatively compact in $G$, and thus Borel measurable. 

In the sequel, we will repeatedly pass to the quotient $ G / Z(G)$
of $G$ and its center $Z = Z(G)$. 
The group $G / Z(G)$ is homogeneous itself.
For the precise details, see  Appendix \ref{sec:homogeneous_quotient}. 

\subsection{Discrete sets} 
A subset $\Lambda \subset G$ of a homogeneous group $G$ is called \emph{relatively separated} if
\[
\rel(\Lambda) := \sup_{x \in G} \# \big( \Lambda \cap B_1 (x) \big) < \infty
\]
and is called \emph{separated} if 
\[
\sep(\Lambda) := \inf_{\lambda \neq \lambda' \in \Lambda} |\lambda^{-1} \lambda'|_G > 0.
\]
A set $\Lambda \subset G$ is called \emph{relatively dense} if there exists an $R>0$ 
such that 
$
G = \bigcup_{\lambda \in \Lambda} B_R (\lambda).
$

A (uniform) \emph{lattice} $\Lambda \subset G$ is a discrete, co-compact subgroup of $G$.
By Malcev's theorem \cite[Theorem 5.1.8]{corwin1990representations}, a nilpotent Lie group $G$ admits a lattice 
only if its Lie algebra has a rational structure.
More generally, any nilpotent Lie group admits a so-called quasi-lattice \cite[Proposition 5.10]{fuehr2007sampling}. 
A set $\Lambda \subset G$ is called a \emph{quasi-lattice} in $G$ 
if there exists a relatively compact Borel set $\Omega \subset G$ 
such that 
$
G = \bigcup_{\lambda \in \Lambda} \lambda \Omega
$
with $\lambda \Omega \cap \lambda' \Omega = \emptyset$ for $\lambda \neq \lambda'$. 
The set $\Omega$ is called the \emph{complement} of $\Lambda$. 

The (homogeneous) \emph{lower} and \emph{upper Beurling density} 
of a discrete set $\Lambda \subset G$ are defined by
\[
D^- (\Lambda) := \liminf_{R \to \infty} \inf_{x \in G} \frac{\# \big( \Lambda \cap B_R (x) \big)}{\mu_G (B_R (e))}
\quad \text{and} \quad
D^+ (\Lambda) := \limsup_{R \to \infty} \sup_{x \in G} \frac{\# \big( \Lambda \cap B_R (x) \big)}{\mu_G (B_R (e))},
\]
respectively. 
For $r > 0$, 
we have  $D^- (D_r (\Lambda)) = r^{-Q} D^- (\Lambda)$ and $D^+ (D_r (\Lambda)) = r^{-Q} D^+ (\Lambda)$. 

A set $\Lambda \subset G$ is relatively separated if, and only if, $D^+ (\Lambda) < \infty$. For a quasi-lattice $\Lambda \subset G$ with complement $\Omega \subset G$, 
we have $D^+ (\Lambda) = D^- (\Lambda) = (\mu_G (\Omega))^{-1}$. 

\subsection{Projective and square-integrable representations} \label{sec:projective}
Let $N$ be a connected, simply connected nilpotent Lie group.
A unitary representation $(\pi, \Hpi)$ of $N$ 
is said to be \emph{square-integrable modulo its centre $Z(N)$}
if there exists a non-zero $g \in \Hpi$ such that
\begin{align} \label{eq:square_integrable}
\int_{N / Z(N)} |\langle g, \pi (x) g \rangle |^2 \; d\mu_{N/Z(N)} (\xdot) < \infty,
\end{align}
where $\xdot = x Z(N)$. If \eqref{eq:square_integrable} is satisfied for some $g \in \Hpi \setminus \{0\}$, 
then it is satisfied for all $g \in \Hpi$. 
We write $\pi \in \SImodZ$ if $(\pi, \Hpi)$ is irreducible and square-integrable modulo $Z(N)$. 
A $\pi \in \SImodZ$ is called a \emph{relative discrete series representation} \cite{moore1973square}. 

Given $\pi \in \SImodZ$, there exists a $d_{\pi} > 0$, 
called the \emph{formal degree} or \emph{formal dimension} of $\pi$, such that the \emph{orthogonality relations}
\begin{align} 
\int_{N/Z(N)} \langle \pi (\xdot) f_1, g_1 \rangle \overline{\langle \pi(\xdot) f_2, g_2 \rangle} \; d\mu_{N/Z(N)} (\xdot) 
= d_{\pi}^{-1} \langle f_1, f_2 \rangle \overline{\langle g_1, g_2 \rangle} \label{ORs}
\end{align}
hold for all $f_1, f_2, g_1, g_2 \in \Hpi$. 

In the sequel, we will often treat a representation $\pi \in \SImodZ$ 
as a projective representation of the quotient $N/Z(N)$. 
A (continuous) \emph{projective representation} $(\overline{\pi}, \Hpi)$ 
of a connected, simply connected nilpotent Lie group $G$
is a strongly continuous map $\overline{\pi} : G \to \mathcal{U} (\Hpi)$ 
satisfying $\overline{\pi} (e) = I$ and such that there exists a continuous 
$\sigma : G \times G \to \mathbb{T}$, called the \emph{cocycle}, 
satisfying $\overline{\pi} (xy) = \sigma(x,y) \overline{\pi} (x) \overline{\pi} (y)$ for $x,y \in G$.

Let $q : N \to N /Z(N)$ denote the quotient map and let
$s : N /Z(N) \to N$ be a continuous cross-section of $q$ such that $q \circ s = \id_{N/Z(N)}$. 
Then, given $\pi \in \SImodZ$, the map
\[
\overline{\pi} : N / Z(N) \to \mathcal{U}(\Hpi), \quad \xdot \mapsto \pi (s(\xdot))
\]
forms a projective representation $(\overline{\pi}, \Hpi)$ of $G := N / Z(N)$ whose representation 
coefficients satisfy
$\langle f, \overline{\pi} (\cdot) g\rangle \in L^2 (G)$ for all $f, g \in \Hpi$. 
A projective representation obtained in this manner is independent 
of the choice of the cross-section and will be referred to as 
a \emph{projective relative discrete series representation}. 
In the sequel, we will simply write $\pi$ for $\overline{\pi}$. 

A vector $g \in \Hpi$ is called a \emph{smooth vector} of a 
relative discrete series representation $(\pi, \Hpi)$ of $N$ if the vector-valued map 
$x \mapsto \pi (x) g$ from $N$ into $\Hpi$ is smooth. 
The space of smooth vectors will be denoted by $\Hsmooth$. 
In particular, given $k \in C_c^{\infty} (N)$ and $h \in \Hpi$, the 
associated \emph{G\aa rding vector} $\pi (k) h := \int_N k(x) \pi (x) h \; d\mu_N (x) \in \Hpi$ 
is a smooth vector. Moreover, by the Dixmier-Malliavin theorem \cite{dixmier1978factorisations}, 
\[
\Hsmooth = \spann \big\{ \pi (k) h \; : \; k \in C_c^{\infty} (N), \; h \in \Hpi \big\} . 
\]
The set of G\aa rding vectors is norm dense in $\Hpi$, and hence so is $\Hsmooth$.  

If $(\pi, \Hpi)$ is a relative discrete series representation of $N$, then 
for any two smooth vectors $g, h \in \Hsmooth$, 
the map $\langle f, \overline{\pi} (\cdot) g \rangle \in \mathcal{S} (G)$, 
where $\mathcal{S} (G)$ denotes the Schwartz space on the quotient group $G = N / Z(N)$, 
e.g, see \cite{corwin1990representations}. 
Consequently, any relative discrete series representation $(\pi, \Hpi)$ is also integrable 
in the sense that there exists a $g \in \Hpi \setminus \{0\}$ such that 
$\langle g, \overline{\pi} (\cdot) g \rangle \in L^1 (N / Z(N))$. 

For more information on projective and relative discrete series representations, 
the interested reader is referred to the books by Wolf  \cite{wolf2007harmonic} 
and Corwin and Greenleaf \cite{corwin1990representations}. 

\subsection{Amalgam spaces}
Let $G$ be a homogeneous group. 
For $F \in L^{\infty}_{\loc} (G)$, the associated (left-sided) \emph{control function}
$F^{\sharp} : G \to \mathbb{C}$ is given by 
$
F^{\sharp} (x) = \esssup_{u \in B_1 (e)} | F(xu)|
$
for $x \in G$. The (left-sided) \emph{Wiener amalgam space} $\WL(G)$ is defined by
\[
\WL(G) := \big\{ F \in L^{\infty}_{\loc} (G) \; : \; F^{\sharp} \in L^1 (G) \big\}
\]
and endowed with the norm $\| F \|_{\WL} := \| F^{\sharp} \|_{L^1}$. 
Instead of taking the supremum over the unit ball $B_1(e)$, one might
take the supremum over an arbitrary compact neighborhood of $e$ and
obtain an  equivalent norm on $W(L^\infty , L^1)(G)$.
See \cite{ho75, fost85} for background on amalgam spaces on the Euclidean space, and \cite{feichtinger1983banach} for their generalization to groups and
norms that measure smoothness.

Similarly, the right-sided control function of an element $F \in L^{\infty}_{\loc} (G)$ 
is defined by $F_{\sharp} (x) := \esssup_{u \in B_1 (e)} |F(ux)|$ and 
the associated amalgam space $\WR(G)$ is endowed 
with the norm $\| F \|_{\WR} := \| F_{\sharp} \|_{L^1}$. 
Note that $\| F \|_{\WR} = \| F^{\vee} \|_{\WL}$, where $F^{\vee} (x) := F(x^{-1})$. 

The (closed) subspaces of $\WL(G)$ and $\WR(G)$ consisting of continuous 
functions are denoted by $\WLC(G)$ and $\WRC(G)$, respectively. 

For technical reasons, we will need some non-standard amalgam spaces
considered in \cite{romero2012characterization}.
The \emph{strong  amalgam} space $\WLst (G)$ (or two-sided amalgam space) is
defined with the control function 
$$
F^\sharp _\sharp (x) = \esssup _{u,v \in B_1(e) } |F(uxv)| =
(F^\sharp)_\sharp (x) = (F_\sharp)^\sharp (x) \, ,
$$
and with norm $\| F \|_{\WRst} := \| (F^\sharp _{\sharp})
\|_{L^1}$. 
By definition, the space $\WLst(G)$ is contained in $\WL(G) \cap \WR(G)$ and
$F \in \WLst(G) $ if and only if $F^{\vee} \in \WLst(G)$, because $G$ is
unimodular. 
For all $x \in G$, it is easy to see that $(F_1 \ast F_2)^{\sharp}_{\sharp} (x) \leq ((F_1)_{\sharp}
\ast (F_2)^{\sharp}) (x) $ and thus it follows that 
\begin{align}
\label{eq_conv_rel}
\WR \ast \WL  \hookrightarrow \WLst.
\end{align}
See \cite[Section 2.4]{romero2012characterization} for similar estimates.

For $\alpha \in \mathbb{R}$, define the weight 
$w_{\alpha} : G \to \mathbb{R}^+, \; x \mapsto (1+|x|_G)^{\alpha}$ 
and equip the Beurling algebra $L^1_{w_{\alpha}} (G)$ with the norm $\| F \|_{L^1_{w_{\alpha}}} := \| w_{\alpha} \cdot F \|_{L^1}$. Then weighted versions of the amalgam spaces are defined as above 
using $L^1_{w_{\alpha}}$ instead of $L^1$.

Identifying the dual space $(C_0 (G))'$ with the space of complex regular Borel measures $\MG$, 
we may identify the dual space $(\WLC(G))'$ with $\WLM(G)$, the space of all 
(locally) complex regular Borel measures $\mu : \mathcal{B} (G) \to \mathbb{C}$ satisfying
\[
\| \mu \|_{\WLM} := \sup_{x \in G} |\mu|(xU) < \infty.
\]
The space $\WLM (G)$ is endowed with the norm $\| \cdot \|_{\WLM}$
and is often called the space of translation-bounded measures. 
If $\Lambda \subset G$ is a relatively separated set, then the measure
$\Sha_{\Lambda} := \sum_{\lambda \in \Lambda} \delta_{\lambda}$ 
belongs to $\WLM (G)$, with $\| \Sha_{\Lambda} \|_{\WLM} \asymp \rel(\Lambda)$. 

\subsection{Coorbit spaces} \label{sec:coorbit}
Let $(\pi, \Hpi)$ be a projective relative discrete series representation
 of a homogeneous group $G$. 
For a fixed non-zero $g \in \Hpi$,  define the associated map $V_g : \Hpi \to L^{\infty} (G)$ by 
$
V_g f (x) :=  \langle f, \pi (x) g \rangle. 
$
The class of \emph{analyzing vectors} $\Api$ is then defined by
$
\Api := \{ g \in \Hpi \; : \; V_g g \in L^1 (G) \}. 
$
The integrability of $(\pi, \Hpi)$ implies that $\Api \neq \emptyset$. 
For a fixed $g \in \Api \setminus \{0\}$, define
\[
\Hone := \{ f \in \Hpi \; : \; V_g f \in L^1 (G) \}
\]
and equip it with the norm $\| f \|_{\Hone} := \| V _g f \|_{L^1}$. 
Let $\Honedual$ denote the anti-dual space of $\Hone$, i.e., 
the space of all conjugate-linear functionals on $\Hone$. 
The associated sesquilinear dual pairing is denoted by 
$\langle \cdot, \cdot \rangle : \Honedual \times \Hone \to \mathbb{C}$. 
The extended representation coefficients are defined by
$
V_g f (x) := \langle f, \pi (x) g \rangle
$
for $f \in \Honedual$ and $g \in \Hone$. 

For $p \in [1,\infty]$ and $g \in \Api \setminus \{0\}$, 
the associated \emph{coorbit space} is defined as the space
\[
\CoLp := \bigg\{ f \in \Honedual \; : \; V_g f \in L^p (G) \bigg\}
\]
equipped with the norm $\| f \|_{\CoLp} := \| V_g f \|_{L^p}$. 

The spaces $\Hone$, $\Honedual$, and $\CoLp$ are $\pi$-invariant Banach spaces
independent of the choice of $g \in \Api \setminus \{0\}$, with equivalent norms for
different choices. Moreover, we have $\Api = \Hone = \Co(L^1(G))$ and $\Hpi = \Co(L^2 (G))$. 
See \cite{christensen1996atomic, feichtinger1989banach1, feichtinger1989banach2} for more details. 

As an auxillary space, we define the closed subspace 
\[
\Conull := \bigg\{ f \in \Co (L^{\infty}(G)) \; : \; V_g f \in C_0 (G) \bigg\}
\]
of $\Co (L^{\infty} (G))$. By duality of coorbit spaces, we then have
$\Conull' = \CoLone$ and $(\CoLone)' = \CoLinfty$ with the duality pairing
$\langle f, h \rangle := \langle V_g f, V_g h \rangle$.

\subsection{Coherent systems and associated operators} \label{sec:coherent}
Let $(\pi, \Hpi)$ be a projective discrete series representation of a homogeneous group $G$. 
In the treatment of coherent systems $\CS$ and their associated operators,  
we will occasionally use the smaller class of \emph{better vectors} $\Bpi \subseteq \Api$, 
defined by
\[
\Bpi := \bigg\{ g \in \Hpi \; : \; V_g g \in \WRC(G) \bigg\}. 
\]
In particular, any smooth vector $g \in \Hsmooth$ of a discrete series representation is in $\Bpi$.

Given $g \in \Bpi$ and a relatively separated set $\Lambda \subset G$,
the \emph{coefficient} and \emph{reconstruction operators} are defined by
\begin{align*}
C_{g, \Lambda} f = \big \{ \langle f, \pi (\lambda) g \rangle \big \}_{\lambda \in \Lambda},
\quad  f \in \CoLinfty
\end{align*}
and 
\begin{align*}
D_{g, \Lambda} c = \sum_{\lambda \in \Lambda} c_{\lambda} \pi (\lambda) g, 
\quad c = \{ c_{\lambda} \}_{\lambda \in \Lambda} \in \ell^{\infty} (\Lambda),
\end{align*}
respectively. 

For $p \in [1,\infty]$,
the maps $C_{g, \Lambda} : \CoLp \to \ell^p (\Lambda)$ 
and $D_{g, \Lambda} : \ell^p (\Lambda) \to \CoLp$ are well-defined and bounded,
with
\begin{align*}
\| C_{g, \Lambda} f \|_{\ell^p} \lesssim \rel(\Lambda) \| f \|_{\CoLp}  \\
\| D_{g, \Lambda} c \|_{\CoLp} \lesssim \rel(\Lambda) \| c \|_{\ell^p},
\end{align*}
where the implicit constants only depend on $g \in \Bpi$, 
see \cite{groechenig1991describing, groechenig2009molecules}. 

The coherent system $\pi (\Lambda) g = \{ \pi(\lambda)g \}_{\lambda \in \Lambda}$ is said to be a \emph{$p$-frame} for $\CoLp$ if
$\| C_{g, \Lambda} f \|_{\ell^p} \asymp \| f \|_{\CoLp}$ for all $f \in \CoLp$, 
while it is called a \emph{$p$-Riesz sequence} in $\CoLp$ if
$\| D_{g, \Lambda} c \|_{\CoLp} \asymp \| c \|_{\ell^p}$ for all $c \in \ell^p (\Lambda)$. 
For $p = 2$, the terminology coincides with the standard definitions for 
frames and Riesz sequences in a Hilbert space.  

We mention the following necessary conditions without proof; 
see \cite{groechenig2015deformation,fuehr2017density} for proofs in similar settings.

\begin{lemma} \label{lem:neccond_geometry}
Let $g \in \Hpi$ and let $\Lambda \subset G$ be a discrete set. 
\begin{enumerate}[(i)]
\item If $\CS$ forms a frame for $\Hpi$, then $\Lambda$ is relatively separated and relatively dense. 
\item If $\CS$ forms a Riesz sequence in $\Hpi$, then $\Lambda$ is separated.
\end{enumerate}
\end{lemma}

\subsection{Universality of frames and Riesz sequences}
We state the following universality result.
Its proof relies on the stability and spectral invariance of 
localized matrices and is deferred to the appendix. 

\begin{theorem} \label{thm:frame_Riesz_universal}
Let $(\pi, \Hpi)$ be a 
projective relative discrete series representation of a homogeneous group $G$. 
Let $g \in \Hsmooth$ and let $\Lambda \subseteq G$ be relatively separated. 
\begin{enumerate}[(i)]
\item If $\CS$ forms a $p$-frame for $\CoLp$ for some $p \in [1,\infty]$, 
then $\CS$ forms a $p$-frame for $\CoLp$ for all $p \in [1,\infty]$. 
\item If $\CS$ forms a $p$-Riesz sequence in $\CoLp$ for some $p \in [1,\infty]$, 
then $\CS$ forms a $p$-Riesz sequence in $\CoLp$ for all $p \in [1,\infty]$. 
\end{enumerate}
\end{theorem}

\begin{remark}
The smoothness condition $g \in \Hsmooth$ in Theorem \ref{thm:frame_Riesz_universal} 
is sufficient for our purposes, but
it can be weakened to the assumption $g \in \Co (L_{w_{\alpha}} (G))$ for $\alpha \geq Q+1$. 
\end{remark}

\section{Stability of frames and Riesz sequences}
This section is devoted to the stability of coherent frames and Riesz sequences 
under weak limits of translates and homogeneous dilations of the index set. 

\subsection{Weak limits of translates}
We start by introducing the notion of weak convergence of sets in the setting 
of a homogeneous group. 

\begin{definition}
Let $G$ be a homogenous group and let $\Lambda \subseteq G$ be arbitrary. 
A sequence $\{ \Lambda_n \}_{n \in \mathbb{N}}$ of subsets $\Lambda_n \subset G$ 
is said to \emph{converge weakly} to $\Lambda$ if for every $R > 0$ and $\varepsilon > 0$, 
there exists an $n_0 \in \mathbb{N}$ such that, for all $n \geq n_0$,
\begin{align} \label{eq:weak_convergence}
\Lambda \cap B_R (e) \subseteq B_{\varepsilon} (e) \Lambda_n 
\quad \text{and} \quad
\Lambda_n \cap B_R (e) \subseteq  B_{\varepsilon} (e) \Lambda. 
\end{align}
The weak convergence of $\{\Lambda_n\}_{n \in \mathbb{N}}$  to $\Lambda$ 
will be denoted by $\Lambda_n \xrightarrow{w} \Lambda$. 
\end{definition}

Given a relatively separated set $\Lambda \subset G$, 
we denote by $W(\Lambda)$ the set of $\Gamma \subset G$ for which there exists a sequence $\{x_n\}_{n \in \mathbb{N}} \subset G$ such that $x_n^{-1} \Lambda \xrightarrow{w} \Gamma$. 
Note that any weak limit $\Gamma \in W(\Lambda)$ is relatively separated, and hence closed. 

\begin{theorem}
Let $(\pi, \Hpi)$ be a 
projective relative discrete series representation of a homogeneous group $G$. 
Let $p \in [1,\infty]$ and let $\Lambda \subseteq G$ be relatively separated. 

If $g \in \Hsmooth$ and $\CS$ is a $p$-frame for $\CoLp$, 
then $\pi (\Gamma) g$ is a $p$-frame for $\CoLp$ for any $\Gamma \in W(\Lambda)$. 
\end{theorem}
\begin{proof}
By Theorem \ref{thm:frame_Riesz_universal}, 
the adjoint map $C_{g, \Lambda} = D_{g, \Lambda}^* : \CoLinfty \to \ell^{\infty} (\Lambda)$ 
of $D_{g, \Lambda} : \ell^1 (\Lambda) \to \Co(L^1 (G))$ is bounded from below, 
hence $C^*_{g, \Lambda} : \ell^1 (\Lambda) \to \CoLone$ 
is surjective by the closed range theorem \cite[Theorem 4.13]{rudin1991functional}. 

We will show that also $C^*_{g, \Gamma} : \ell^1 (\Gamma) \to \CoLone$ 
is surjective. Fix $f \in \CoLone$ and  let $\{x_n\}_{n \in \mathbb{N}}$ be a sequence in $G$ such that $x_n^{-1} \Lambda \xrightarrow{w} \Gamma$. 
For $n \in \mathbb{N}$, define $\Lambda_n := x_n^{-1} \Lambda$. 
Then, for fixed $n \in \mathbb{N}$, 
the map $C^*_{g,  \Lambda_n} : \ell^1 (\Lambda_n) \to \CoLone$ is also a bounded surjection. 
Moreover, by the open mapping theorem, the maps $C^*_{g, \Lambda_n}$ 
have bounds on preimages independent of $n \in \mathbb{N}$. 
Thus
there exists a sequence $\{c^{(n)}_{\lambda} \}_{\lambda \in \Lambda_n}$ 
satisfying $\| c^{(n)} \|_{\ell^1} \lesssim 1$, with a constant independent of $n$, 
and such that $f = \sum_{\lambda \in \Lambda_n} c_{\lambda}^{(n)} \pi (\lambda) g$,
with norm convergence in $\CoLone$. 

Define $\mu_n := \sum_{\lambda \in \Lambda_n} c^{(n)}_{\lambda} \delta_{\lambda} \in M(G)$, 
and note that $\|\mu_n \|_M := |\mu_n|(G) = \| c^{(n)} \|_{\ell^1} \lesssim 1$,
with a bound independent of $n$.
By Banach-Alaoglu's theorem, there exists a subsequence, 
also denoted by $\{ \mu_n \}_{n \in \mathbb{N}}$, 
such that $\mu_n \to \mu \in M(G)$ in the vague topology $\sigma(M, C_0)$. 
Since $\supp (\mu_n) \subseteq x_n^{-1} \Lambda$ and $x_n^{-1} \Lambda \xrightarrow{w} \Gamma$ 
by assumption, it follows that 
$\supp (\mu) \subseteq \overline{\Gamma} = \Gamma$ 
by the defining condition \eqref{eq:weak_convergence}. 
Consequently, we can write $\mu = \sum_{\lambda \in \Gamma} c_{\lambda} \delta_{\lambda}$ 
for some sequence $c = \{ c_{\lambda}\}_{\lambda \in \Gamma}$ satisfying
$\| c \|_{\ell^1} = |\mu| (G) \leq \liminf_{n \to \infty} \|\mu_n \|_M \lesssim 1$. 
Using this, we define $f' \in \Co(L^1)(G)$ by
\[ f' := \sum_{\lambda \in \Gamma} c_{\lambda} \pi (\lambda) g. \]
Since $V_g \pi (x) g \in \WLC(G) \hookrightarrow  C_0 (G)$ for arbitrary, but fixed $x \in G$, 
a direct calculation yields
\begin{align*}
\langle f, \pi (x) g \rangle 
= \sum_{\lambda \in \Lambda_n} c_{\lambda}^{(n)} \overline{V_g \pi(x)
  g } 
= \int_G \overline{V_g \pi(x) g} \; d\mu_n 
\to 
\int_G \overline{V_g \pi(x) g } \; d\mu
= \langle f', \pi(x) g \rangle,
\end{align*}
hence $f = f'$. 
This shows that $C_{g, \Gamma}^* : \ell^1 (\Gamma) \to \CoLone$ is a surjection. 
Another application of the closed range theorem yields that 
$C_{g, \Gamma} : \CoLinfty \to \ell^{\infty} (\Gamma)$ is bounded 
from below, thus showing that $\pi (\Gamma) g$ forms an $\infty$-frame for 
$\CoLinfty$, and, more generally, a $p$-frame for $\CoLp, p \in [1, \infty],$
by Theorem \ref{thm:frame_Riesz_universal}. 
\end{proof}

The following result provides the stability of Riesz sequences under weak limits 
of translates. 

\begin{theorem}
Let $(\pi, \Hpi)$ be a 
projective relative discrete series representation of a homogeneous group $G$. 
Let $p \in [1,\infty]$ and let $\Lambda \subseteq G$ be separated. 

If $g \in \Hsmooth$ and $\CS$ is a $p$-Riesz sequence in $\CoLp$, 
then $\pi (\Gamma) g$ is a $p$-Riesz sequence in $\CoLp$ for any $\Gamma \in W(\Lambda)$. 
\end{theorem}
\begin{proof}
The map $C^*_{g, \Lambda} :  \ell^{\infty} (\Lambda) \to \CoLinfty)$ 
is bounded from below by Theorem \ref{thm:frame_Riesz_universal}. 
Thus, by the closed range theorem, the map $C_{g, \Lambda} : \CoLone \to \ell^1 (\Lambda)$ 
is surjective. 

Consider a sequence $\Lambda_n = x_n^{-1} \Lambda \xrightarrow{w} \Gamma$. 
For an arbitrary, but fixed $\gamma \in \Gamma$, choose a sequence $\{ \gamma_n \}_{n \in \mathbb{N}}$ 
of points $\gamma_n \in \Lambda$ such that $x_n^{-1} \gamma_n \to \gamma$. 
Since $C_{g, \Lambda} : \CoLone \to \ell^1 (\Lambda)$ is open and surjective, 
for any $c \in \ell^1 (\Lambda)$ with $\| c \|_{\ell^1} = 1$ there exists an 
$f \in \CoLone$ with $\| f \|_{\CoLone} \lesssim 1$ such that $c = C_{g, \Lambda} f$. 
Moreover, for each $n \in \mathbb{N}$, there exists an $f_n \in \CoLone$ 
with $\| f_n \|_{\Co(L^1(G))} \lesssim 1$ such that $c = C_{g, \Lambda_n} f_n$ for $\Lambda_n := x_n^{-1} \Lambda$.  
In particular, there exists an $h_n \in \CoLone$ with $\|V_g h_n \|_{L^1} \lesssim 1$, 
and $V_g h_n (x_n^{-1} \gamma_n) = 1$ and $V_g h_n = 0$ on $x_n^{-1} \Lambda \setminus \{ x_n^{-1} \gamma_n \}$. 
By passing to a subsequence if necessary, it may be assumed that $h_n \to h$ in $\sigma (\CoLone, \Conull)$
for some $h \in \CoLone$ with $\| h \|_{\CoLone} \lesssim 1$. 
By \cite[Theorem 4.1]{feichtinger1989banach1}, 
it follows that $V_g h_n \to V_g h$ with uniform convergence on compacta in $G$. 
Hence
\[ V_g h(\gamma) = \lim_{n \to \infty} V_g h_n (x_n^{-1} \gamma_n) = 1.\] 

Similarly, for $\gamma' \in \Gamma \setminus \{\gamma\}$,
there exists a sequence $\{ \gamma'_n \}_{n \in \mathbb{N}}$ 
in $\Lambda$ such that $x_n^{-1} \gamma'_n \to \gamma'$ 
and $V_g h_n (x_n^{-1} \gamma'_n) = 0$ for sufficiently large $n$,
yielding that $V_g h (\gamma') = 0$. 

Combining the above, it follows that for each $\gamma \in \Gamma$, 
there exists a function $h_{\gamma} \in \CoLone$ with $\| h_{\gamma} \|_{\CoLone} \lesssim 1$, 
and $V_g h_{\gamma} (\gamma) = 1$ and $V_g h_{\gamma} \equiv 0$ on $\Gamma \setminus \{\gamma\}$. 
Thus, for a fixed $c \in \ell^1 (\Gamma)$, defining $f \in \CoLone$ by
\[ f := \sum_{\gamma \in \Gamma} c_{\gamma} h_{\gamma}  \]
gives $C_{g, \Gamma} f = c$, which shows that $C_{g, \Gamma} : \CoLone \to \ell^1 (\Gamma)$ is surjective. 
Consequently $C_{g, \Gamma}^* : \ell^{\infty} (\Gamma) \to \CoLinfty$ is bounded below, 
and $\pi (\Gamma) g$ forms an $\infty$-Riesz sequence in $\CoLinfty$, 
and hence a $p$-Riesz sequence in $\CoLp$ by Theorem \ref{thm:frame_Riesz_universal}. 
\end{proof}

\subsection{Stability under dilations} 
In this section we prove the stability of coherent frames and Riesz sequences under 
dilations. For this, we use the following relation between dilations and weak limits of translates. 
 
\begin{lemma} \label{lem:dilation_subsequence}
Let $G$ be a homogeneous group. 
Let $\Lambda \subset G$ be relatively separated, let $\{x_n\}_{n \in
  \mathbb{N}}\subseteq G$ be arbitrary and let $\{r_n\}_{n \in \mathbb{N}} \subseteq \mathbb{R}^+$  be such that
$\lim _{n\to \infty } r_n = 1$. Then there exists a subsequence of $
\{x_{n}^{-1}  D_{r_{n}}(\Lambda)\}_{n \in \mathbb{N}}$ that 
 converges weakly to a relatively separated set $\Gamma \subset G$. 
Moreover, for any sequence $x_n^{-1} D_{r_n} (\Lambda) \xrightarrow{w} \Gamma$, 
the limit $\Gamma \in W(\Lambda)$. 
\end{lemma}

\begin{proof}
Throughout the proof, we define $\Lambda _n := D_{r_n}(\Lambda ) $ for $n \in \mathbb{N}$. 
The proof is divided into two steps: 

\textbf{Step 1.} \emph{(Existence of a subsequence).}
 We use 
the uniform relative separation 
\[
\limsup_{n \in \mathbb{N}} \rel(x_n^{-1} \Lambda_n ) = \limsup_{n \in \mathbb{N}} \rel(\Lambda_n) < \infty, 
\] 
and  the norm equivalence $\| \Sha_{x_n^{-1} \Lambda_n} \|_{\WLM}
\asymp \rel(\Lambda_n)$, and obtain the existence of a subsequence $\{
\Sha_{x_{n_k}^{-1} \Lambda_{n_k}} \}_{k \in \mathbb{N}}$  
that converges to an element $\mu \in \WLM (G)$ in the weak$^*$-topology
$\sigma(\WLM(G), \WLC(G))$. 
Since $\Sha_{x_n^{-1} \Lambda_n} (E) \in \mathbb{N} \cup \{\infty\}$ for any Borel set $E \in \mathcal{B}(G)$, 
the convergence $\Sha_{x_n^{-1} \Lambda_{n_k}} \to \mu$ yields that 
\[
x_n^{-1} \Lambda_{n_k} = \supp(\Sha_{x_n^{-1} \Lambda_{n_k}}) \xrightarrow{w} \supp(\mu) =: \Gamma,
\] 
See \cite[Section 4]{groechenig2015deformation} for details.
The set $\Gamma$ is relatively separated, and hence closed. 

\textbf{Step 2.} \emph{(Weak limit of translates)}. 
By the above, there exists a sequence $\{ x_n^{-1} \Lambda_n \}_{n \in \mathbb{N}}$ 
converging weakly to a relatively separated set $\Gamma \subset G$. 
For $n \in \mathbb{N}$, write 
\[ x_n^{-1} \Lambda_n = D_{r_n} \big((D_{r_n}^{-1} (x_n^{-1})) \Lambda\big). \]
By passing to a subsequence, it may be assumed that 
$D_{r_n}^{-1}(x_n^{-1}) \Lambda \xrightarrow{w} \Gamma'$. 
We claim that $\Gamma = \Gamma'$. 

For the inclusion $\Gamma' \subseteq \Gamma$, 
let $R > 0$ and $\varepsilon \in (0,1]$, and
take a $z \in D_{r_n}^{-1} (x_n^{-1}) \Lambda \cap B_R (e)$. 
Then $z = D_{r_n}^{-1}(x_n^{-1}) \lambda$ for some $\lambda \in \Lambda$ 
and $|D_{r_n}^{-1} (x_n^{-1})\lambda|_G \leq R$. 
Choose $n_0 \in \mathbb{N}$ such that 
$|x_n^{-1} D_{r_n} (\lambda) z^{-1} |_G = |D_{r_n}^{-1}(x_n^{-1}) \lambda (D_{r_n} (\lambda))^{-1} x_n |_G < \varepsilon$ for all $n \geq n_0$.
Then, if $n \geq n_0$, it follows that 
$z^{-1} \in (D_{r_n} (\lambda))^{-1} x_n B_{\varepsilon} (e)$,
and hence 
\[
D_{r_n}^{-1}(x_n^{-1}) \Lambda \cap B_R (e) \subseteq   B_{\varepsilon} (e) x_n^{-1} \Lambda_n. 
\]
Using that $D_{r_n}^{-1}(x_n^{-1}) \Lambda \xrightarrow{w} \Gamma'$ 
and $x_n^{-1} \Lambda_n \xrightarrow{w} \Gamma$, 
it follows that $\Gamma' \subseteq \Gamma$. 

For the converse, 
let again $R > 0$ and $\varepsilon \in (0,1]$. 
Take an arbitrary $z \in x_n^{-1} \Lambda_n \cap B_R (e)$. 
Then $z = x_n^{-1} D_{r_n} (\lambda)$ for some $\lambda \in \Lambda$ 
and $|x_n^{-1} D_{r_n} (\lambda) |_G \leq R$. 
Hence there exists an $R' > 0$ such that $|D_{r_n}^{-1}z|_G=  |D_{r_n}^{-1} (x_n^{-1}) \lambda|_G \leq R'$. 
Consequently, we can choose an $n_0 \in \mathbb{N}$ such that 
$|D_{r_n}^{-1} (x_n^{-1}) \lambda z^{-1} |_G = |D_{r_n}^{-1} (x_n^{-1}
) \lambda D_{r_n} (\lambda)^{-1} x_n  |_G < \epsilon$ for all $n \geq n_0$. 
Thus, if $n \geq n_0$, then $z^{-1} \in \lambda^{-1} D_{r_n}^{-1} (x_n) B_{\epsilon} (e)$, 
which shows that
\[
x_n^{-1} \Lambda_n \cap B_R (e) \subseteq B_{\epsilon} (e) D^{-1}_{r_n} (x_n^{-1}) \Lambda.
\]
Since $x_n^{-1} \Lambda_n \xrightarrow{w} \Gamma$ and 
$D_{r_n}^{-1} (x_n^{-1}) \Lambda \xrightarrow{w} \Gamma'$, it follows that 
$\Gamma \subseteq \Gamma'$. 

Combining the obtained inclusions yields $\Gamma = \Gamma'$, 
and thus $\Gamma \in W(\Lambda)$. 
\end{proof}

The following result shows the stability of coherent
frames  with respect to dilations. 

\begin{theorem} \label{thm:frame_stability} 
Let $(\pi, \Hpi)$ be a 
projective relative discrete series representation of a homogeneous group $G$. 
Let $g \in \Hsmooth$ and let $\Lambda \subset G$ be a discrete set. 
Let $\{\Lambda_n\}_{n \in \mathbb{N}} = \{D_{r_n} (\Lambda) \}_{n \in \mathbb{N}}$
for some $\{r_n \}_{n \in \mathbb{N}} \subset \mathbb{R}^+$ with $r_n \to 1$. 
If $\CS$ forms a frame for $\Hpi$,
then $\pi (\Lambda_n) g$ forms a frame for $\Hpi$ for all sufficiently large $n$. 
\end{theorem}
\begin{proof}
By Lemma \ref{lem:neccond_geometry}, the set $\Lambda \subset G$ is relatively separated 
and relatively dense. 
Arguing indirectly, assume that $\pi (\Lambda_n) g$ fails to be  
a frame for sufficiently large $n$. Then, by passing to a subsequence, we
assume that $\pi (\Lambda_n) g$ fails to be a frame for all $n \in \mathbb{N}$. 
In particular, every $\pi (\Lambda_n) g$ fails to be an $\infty$-frame for $\CoLinfty$ 
by Theorem \ref{thm:frame_Riesz_universal}. 
Hence, there exists a sequence $\{ f_n \}_{n \in \mathbb{N}}$ in $\CoLinfty$ 
such that $\|V_g f_n \|_{L^{\infty}} = 1$ and
\[
\| C_{g, \Lambda_n} f_n \|_{\ell^{\infty} (\Lambda_n)} 
= \sup_{\lambda \in \Lambda} |V_g f_n (D_{r_n} (\lambda)) | \to 0 \hspace{5pt} \mbox{ as } \hspace{5pt} n \to \infty.
\]
Given $n \in \mathbb{N}$, let $x_n \in G$ be such that $|V_g f_n (x_n)| \geq 1/2$
and let $h_n := \pi(x_n^{-1}) f_n$. 
By passing to a subsequence, we assume
that $h_n \to h \in \CoLinfty$ in the $w^*$-topology $\sigma(\CoLinfty, \CoLone)$.  
Then $h \neq 0$ since $|V_g h_n (0)| = |V_g f_n (x_n)| \geq 1/2$ implies $|V_g h(0)|\geq 1/2$. 

By Lemma  \ref{lem:dilation_subsequence}, it may be assumed that 
$x_n^{-1} \Lambda_n \xrightarrow{w} \Gamma$ for 
some relatively separated set $\Gamma\in W(\Lambda)$.
Let $\gamma \in \Gamma$ be arbitrary and choose a sequence
$\{\gamma_n\}_{n \in \mathbb{N}} \subseteq \Lambda$ such that $x_n^{-1} D_{r_n} (\gamma_n) \to \gamma$ 
as $n \to \infty$. By \cite[Theorem 4.1]{feichtinger1989banach1}, 
the convergence $V_g h_n \to V_g h$ is uniform on compact subsets of $G$, so
\[
|V_g h(\gamma)| = \lim_{n \to \infty} |V_g h_n (x_n^{-1} D_{r_n} (\gamma_n)) |
= \lim_{n \to \infty} |V_g f_n (D_{r_n} (\gamma_n))| = 0. 
\]
Since $\gamma \in \Gamma$ was arbitrary, this implies that
 $V_g h \equiv 0$ on $\Gamma \in W(\Lambda)$. 
Thus, $\CS$ does not form an $\infty$-frame for $\CoLinfty$ and, by Theorem \ref{thm:frame_Riesz_universal}, 
neither for $\Hpi$, which is a contradiction.
\end{proof}

The following result is the analogue of Theorem \ref{thm:frame_stability} for Riesz sequences. 

\begin{theorem} \label{thm:riesz_stability}
Let $(\pi, \Hpi)$ be a 
projective relative discrete series representation of a homogeneous group $G$. 
Let $g \in \Hsmooth$ and let $\Lambda \subset G$ be a discrete set.
Let $\{\Lambda_n\}_{n \in \mathbb{N}} = \{D_{r_n} (\Lambda) \}_{n \in \mathbb{N}}$
for some $\{r_n \}_{n \in \mathbb{N}} \subset \mathbb{R}^+$ with $r_n \to 1$. 
If $\CS$ forms a Riesz sequence in $\Hpi$, 
then $\pi (\Lambda_n) g$ forms a Riesz sequence in $\Hpi$ for all sufficiently large $n$. 
\end{theorem}
\begin{proof}
The set $\Lambda \subset G$ is separated by Lemma \ref{lem:neccond_geometry}. 
Arguing indirectly, assume that  $\pi(\Lambda_n) g$ 
fails to be a Riesz sequence for all $n \in \mathbb{N}$. 
Then there exist sequences $c^{(n)} \in \ell^{\infty} (\Lambda_n)$ 
with $\| c^{(n)} \|_{\ell^{\infty}} = 1$ and such that
\[
\|D_{g, \Lambda_n} c^{(n)} \|_{\Co (L^{\infty})}
= \bigg\| \sum_{\lambda \in \Lambda_n} c^{(n)}_\lambda \pi (\lambda) g \bigg\|_{\CoLinfty} 
\to 0  \quad \text{ as } \quad n \to \infty.
\]
We will next construct a subsequence $(n_k)_{k \in \mathbb{N}} \subseteq \mathbb{N}$, 
corresponding points $\lambda_{n_k} \in \Lambda_{n_k}$, a separated set $\Gamma \subseteq G$
and a non-zero $c \in \ell^{\infty} (\Gamma)$ such that
$
\lambda_{n_k}^{-1} \Lambda_{n_k}  \xrightarrow{w} \Gamma
$
as $k \to \infty$, and 
\begin{align} \label{eq:linear_independence}
\sum_{\gamma \in \Gamma} c_{\gamma} \pi (\gamma) g = 0.
\end{align}

For $n \in \mathbb{N}$,  choose a point $\lambda_n \in \Lambda_n$ 
such that $|c_{\lambda_n}^{(n)} | \geq 1/2$. 
Write $\theta_{\lambda, n} \in \mathbb{T}$ for the cocycle satisfying
$\theta_{\lambda, n} \pi(\lambda_n^{-1} \lambda) = \pi(\lambda_n^{-1}) \pi(\lambda)$.
Define the measure
$\mu_n := \sum_{\lambda \in \Lambda_n} \theta_{\lambda, n} c^{(n)}_{\lambda} \delta_{\lambda^{-1}_n \lambda}$. 
Since $|\mu_n| \leq \|c^{(n)} \|_{\ell^{\infty}} |\Sha_{\lambda_n^{-1} \Lambda_n}|$,
it follows that $\| \mu_n \|_{\WLM} \lesssim \rel(\Lambda_n) \|c^{(n)} \|_{\ell^{\infty}} \lesssim 1$. 
By passing to a subsequence, we may assume that
$\{ \mu_n \}_{n \in \mathbb{N}}$ converges to some $\mu \in \WLM(G)$ in $\sigma \bigl( \WLM(G), \WLC(G) \bigr)$. 
In addition, we may assume that $\lambda_n^{-1} \Lambda_n^{-1} \xrightarrow{w} \Gamma$ for some relatively separated set $\Gamma$, 
which necessarily satisfies $\supp(\mu) \subseteq \Gamma$
- see \cite[Section 4]{groechenig2015deformation} for similar claims.
It follows that $\mu = \sum_{\gamma \in \Gamma} c_{\gamma} \delta_{\gamma}$ 
for some $c \in \ell^{\infty} (\Gamma)$. To see that $c \neq 0$, let 
$ r := \inf_{n \in \mathbb{N}} \sep(\Lambda_n). $
Then $r > 0$ since $r_n \to 1$ and $\sep(\Lambda) > 0$ by assumption. 
Hence
$B_{r/2} (\lambda_n) \cap \Lambda_n = \{\lambda_n\}$ for $n \in \mathbb{N}$. 
Choosing a bump function $\varphi \in C(G)$ supported on $B_{r/2} (e)$ with $\varphi (e) = 1$
gives
\[
\bigg| \int_G \varphi \; d\mu \bigg| = \lim_{n \to \infty} \bigg| \int_G \varphi \; d\mu_n \bigg|
= \lim_{n \to \infty} |c_{\lambda_n}^{(n)} | \geq 1/2,
\]
thus $c \neq 0$, as claimed. 
Lastly, to show \eqref{eq:linear_independence}, it
suffices to show that $V_g \big(\sum_{\gamma \in \Gamma} c_{\gamma} \pi (\gamma) g \big) = 0$. 
But since $V_g \pi (x) g \in \WLC(G)$ for any $x \in G$, it follows that 
\begin{align*}
\bigg| \bigg\langle \sum_{\gamma \in \Gamma} c_\gamma \pi(\gamma) g, \pi(x) g \bigg\rangle \bigg|
&= \bigg| \sum_{\gamma \in \Gamma} c_{\gamma} \overline{V_g \pi(x) g(\gamma) } \bigg|
= \lim_{n \to \infty} \bigg| \int_G \overline{V_g \pi(x) g} \; d\mu_n \bigg| \\
&= \lim_{n \to \infty} \bigg| \bigg \langle \sum_{\lambda \in \Lambda_n} \theta_{\lambda, n} c_{\lambda}^{(n)} \pi(\lambda_n^{-1} \lambda) g, \pi(x) g \bigg\rangle \bigg| \\
&\leq \lim_{n \to \infty} \bigg\| \pi(\lambda_n^{-1}) \sum_{\lambda \in \Lambda_n} c_{\lambda}^n \pi (\lambda) g \bigg\|_{\Co (L^{\infty}(G))} \| g \|_{\Co (L^1(G))} = 0
\end{align*}
for all $x \in G$. This shows \eqref{eq:linear_independence}. 

Since $\Gamma \in W(\Lambda)$ by Lemma \ref{lem:dilation_subsequence}, 
it follows that $\CS$ does not form a Riesz sequence, 
which contradicts the assumption. 
\end{proof}

\section{Balian-Low type theorems}
In this section we obtain Balian-Low type theorems for coherent systems 
forming a frame or Riesz sequence.

\subsection{Necessary density conditions} 

The following necessary density condition for frames and Riesz sequences, with respect to the homogeneous Beurling density, are well-known, e.g.,  see \cite[Theorem 5.3]{fuehr2017density}. 

\begin{theorem} \label{thm:nonstrict_density}
Let $(\pi, \Hpi)$ be a 
projective relative discrete series representation of a homogeneous group $G$. 
Let $g \in \Bpi$ and let $\Lambda \subset G$ be a discrete set. 
\begin{enumerate}[(i)]
\item If $\CS$ forms a frame for $\Hpi$, then $D^- (\Lambda) \geq d_{\pi}$.
\item If $\CS$ forms a Riesz sequence in $\Hpi$, then $D^+ (\Lambda) \leq d_{\pi}$. 
\end{enumerate}
\end{theorem}
\begin{proof}
The result follows from \cite[Corollary 4.1]{fuehr2017density}
applied to the metric measure space $(G, d_G, \mu_G)$ and the 
reproducing kernel Hilbert space
$
V_g (\Hpi) := \big\{ V_g f \; : \; f \in \Hpi \big\}. 
$
The hypotheses of \cite[Corollary 4.1]{fuehr2017density} 
on the metric measure space $(G, d_G, \mu_G)$ are easily verified, 
using the homogeneity of the Haar measure $\mu_G$,
and the reproducing kernel satisfies the homogeneous approximation property by 
\cite{groechenig2008homogeneous}. 
\end{proof}

\begin{remark}
\begin{enumerate}[(a)]
\item Both the formal dimension $d_{\pi}$ and the densities $D^+, D^-$ depend on 
the choice of the Haar measure $\mu_G$, but the quotients $D^+(\Lambda) / d_{\pi}$ 
and $D^-(\Lambda) / d_{\pi}$ 
do not. 
\item The densities $D^+, D^-$ do not depend on the choice of the homogeneous norm 
as can be shown by adapting the original arguments by Landau \cite{landau1967necessary}. 
See \cite{hoefler2014necessary} for the details. 
\item The formal dimension of a $\pi \in \SImodZ$ can be computed explicitly 
by use of the Pfaffian polynomials. The interested reader is referred to 
\cite{moore1973square, corwin1990representations}
for the details. 
\end{enumerate}
\end{remark}

\subsection{Strict  density inequalities}
In order to prove Theorem \ref{thm:balianlow_homogeneous}, 
we will use the following two auxillary results. 

\begin{lemma} \label{lem:smooth_dense}
The space of smooth vectors $\Hsmooth$ is norm dense in $\Co(L^1 (G))$. 
\end{lemma}
\begin{proof}
Let $f \in \Co(L^1 (G))$ be arbitrary. Fix a $g \in \Hsmooth \setminus \{0\}$. 
By the atomic decomposition result \cite[Theorem 6.1]{feichtinger1989banach1}, 
there exists a relatively separated set $\Lambda = \{ \lambda_i \}_{i \in \mathbb{N}} \subset G$ and a sequence
 $\{c_{i}\}_{i \in \mathbb{N}} \in \ell^1 (\Lambda)$ such that
$ f = \sum_{i \in \mathbb{N}} c_{i} \pi (\lambda_i) g$. 
For $n \in \mathbb{N}$, define $f_n = \sum_{i = 1}^n c_i \pi(\lambda_i) g$. 
Then $f_n \in \Hsmooth$, and $f_n \to f$ in $\Co (L^1 (G))$ as $n \to \infty$. 
\end{proof}

\begin{proposition} \label{prop:perturbation}
Let $(\pi, \Hpi)$ be a 
projective relative discrete series representation of a homogeneous group $G$. 
Let $g \in \Co (L^1 (G))$ and let $\Lambda \subset G$ be discrete. Then
\begin{enumerate}[(i)]
\item If $\pi (\Lambda) g$ forms a frame for $\Hpi$, then there exists 
$\widetilde{g} \in \Hsmooth$ such that $\pi (\Lambda) \widetilde{g}$ 
forms a frame for $\Hpi$.
\item If $\pi (\Lambda) g$ forms a Riesz sequence in $\Hpi$, 
then there exists $\widetilde{g} \in \Hsmooth$ such that $\pi (\Lambda) \widetilde{g}$ 
forms a Riesz sequence in $\Hpi$.  
\end{enumerate}
\end{proposition}
\begin{proof}
Let $\widetilde{g} \in \Hsmooth$ be arbitrary. 
Throughout the proof, fix a vector $h \in \Bpi \setminus \{0\}$ such that
$V_h : \Hpi \to L^2 (G)$ forms an isometry.  

(i) Let $f \in \Hpi$. By \cite[Lemma 2]{romero2012characterization}, 
for any $F \in L^2 (G)$ and $H \in W^R(L^{\infty}, L^1)$, 
it follows that $\| (\langle F, H(\lambda^{-1} \cdot) )_{\lambda} \|_{\ell^2 (\Lambda)} \lesssim \|F\|_{L^2} \| H \|_{W^R(L^{\infty}, L^1)}$.
Hence, there exists $C>0$ such that
\begin{align*}
\big\| (C_{g, \Lambda} - C_{\widetilde{g}, \Lambda} ) f \big\|_{\ell^2 (\Lambda)} 
&= \big\| \langle f, \pi(\lambda) [g- \widetilde{g}] \rangle_{\Hpi} \big\|_{\ell^2 (\Lambda)} 
= \big\| \langle V_h f ,V_h (\pi(\lambda) [ g- \widetilde{g}]) \rangle_{L^2 (G)} \big\|_{\ell^2 (\Lambda)} \\
&\leq C \| V_h f \|_{L^2 (G)} \big\| V_h (g - \widetilde{g}) \big\|_{W^R (L^{\infty}, L^1)} \\
&= C \big\| V_h (g-\widetilde{g}) \big\|_{W^R (L^{\infty}, L^1)} \| f \|_{\Hpi}. 
\end{align*}
The pointwise estimate $|V_h (g - \widetilde{g}) | \leq |V_h (g - \widetilde{g}) | \ast |V_h h|$ 
and the convolution relation $L^1 (G) \ast W^R(L^{\infty}, L^1) \hookrightarrow W^R(L^{\infty}, L^1)$
yield that
\begin{align} \label{eq:conv_repro}
\| V_h (g - \widetilde{g}) \big\|_{W^R (L^{\infty}, L^1)} \leq \| V_h (g - \widetilde{g}) \|_{L^1} \| V_h h \|_{W^R (L^{\infty}, L^1)}. 
\end{align}
Thus $\big\| (C_{g, \Lambda} - C_{\widetilde{g}, \Lambda} ) f \big\|_{\ell^2 (\Lambda)} \leq C \| V_h (g - \widetilde{g}) \|_{L^1 (G)} \| V_h h \|_{W^R (L^{\infty}, L^1)} \| f \|_{\Hpi}$. 

Suppose  $\CS$ forms a frame for $\Hpi$ satisfying $\|C_{g, \Lambda} f \|_{\ell^2 (\Lambda)} \geq A\|f\|_{\Hpi}$ for all $f \in \Hpi$. 
Lemma \ref{lem:smooth_dense} yields a $\widetilde{g} \in \Hsmooth$ such that 
$K := C \| V_h h \|_{W^R (L^{\infty}, L^1)} \| V_h (g - \widetilde{g}) \|_{L^1} < A$. 
Hence
\begin{align*}
\| C_{\widetilde{g}, \Lambda} f \|_{\ell^2 (\Lambda)} 
\geq \| C_{g, \Lambda} f \|_{\ell^2 (\Lambda)} - \| (C_{g, \Lambda} - C_{\widetilde{g}, \Lambda}) f \|_{\ell^2 (\Lambda)} 
\geq (A - K ) \| f \|_{\Hpi}
\end{align*}
for all $f \in \Hpi$, which shows (i). 

(ii) 
Let $c = \{c_{\lambda}\}_{\lambda \in \Lambda} \in \ell^2 (\Lambda)$. 
By \cite[Theorem 5.2]{feichtinger1989banach1}, 
for any $H \in W^R (L^{\infty}, L^1)$, it holds that
$\| \sum_{\lambda \in \Lambda} c_{\lambda} H(\lambda^{-1} \cdot) \|_{L^2} \lesssim \| c \|_{\ell^2 (\Lambda)} \| H \|_{W^R (L^{\infty}, L^1)}$. 
Hence, there exists $C > 0$ such that
\begin{align*}
\bigg\| \big(D_{g, \Lambda} - D_{\widetilde{g}, \Lambda} \big) c \bigg\|_{\Hpi} 
&= \bigg\| \sum_{\lambda \in \Lambda} c_{\lambda} \pi(\lambda) [g - \widetilde{g}] \bigg\|_{\Hpi} 
=   \bigg\| \sum_{\lambda \in \Lambda} c_{\lambda} V_h (\pi (\lambda) [g - \widetilde{g}]) \bigg\|_{L^2} \\
&\leq C  \big\| V_h (g - \widetilde{g}) \big\|_{W^R (L^{\infty}, L^1)}  \| c \|_{\ell^2 (\lambda)} \\
&\leq C \|V_h h \|_{W^R (L^{\infty}, L^1)} \big\| V_h (g - \widetilde{g}) \big\|_{L^1}  \| c \|_{\ell^2 (\lambda)}
\end{align*}

Suppose $\CS$ forms a Riesz sequence with $\|D_{g, \Lambda} c \|_{\Hpi} \geq A \| c \|_{\ell^2}$
for all $c \in \ell^2 (\Lambda)$.  
By Lemma \ref{lem:smooth_dense}, there exists $\widetilde{g} \in \Hsmooth$ such that 
$K := C \|V_h h \|_{W^R (L^{\infty}, L^1)} \| V_h (g - \widetilde{g}) \|_{L^1} < A$. 
Hence
\begin{align*}
\big\| D_{\widetilde{g}, \Lambda} c \big\|_{\Hpi} \geq \big\| D_{g, \Lambda} c \big\|_{\Hpi} - \big\| (D_{g, \Lambda} - D_{\widetilde{g}, \Lambda}) c \big\|_{\Hpi} \geq (A - K) \| c \|_{\ell^2 (\Lambda)} 
\end{align*}
for all $c \in \ell^2 (\Lambda)$. This completes the proof. 
\end{proof}

We now prove Theorem \ref{thm:balianlow_homogeneous}, 
which asserts that the density inequalities in Theorem \ref{thm:nonstrict_density} are strict. 

\begin{proof}[Proof of Theorem \ref{thm:balianlow_homogeneous}]
(i) We argue indirectly and  assume  that $\CS$ is a 
frame with $D^- (\Lambda) = d_{\pi}$. Then, by Proposition \ref{prop:perturbation}, 
there exists a $\widetilde{g} \in \Hsmooth$ such that $\pi (\Lambda) \widetilde{g}$ 
forms a frame for $\Hpi$. 
 Let $\{\Lambda_n\}_{n \in \mathbb{N}} = \{D_{r_n} (\Lambda)\}_{n \in \mathbb{N}}$
with $r_n > 1$ for all $n \in \mathbb{N}$ and $r_n \to 1$ as $n \to \infty$. 
Then, by Theorem \ref{thm:frame_stability}, there exists an $n_0 \in \mathbb{N}$ 
such that for all $n \geq n_0$, the system $\pi (\Lambda_n) \widetilde{g}$ is a frame for $\Hpi$. But 
$D^- (\Lambda_n) = r_n^{-Q} D^- (\Lambda) < d_{\pi}$, 
which contradicts Theorem \ref{thm:nonstrict_density}(i). 

The proof of (ii) is similar.  
Assume that $\CS$ forms a Riesz sequence in $\Hpi$ with $D^+ (\Lambda) = d_{\pi}$. 
By Proposition \ref{prop:perturbation}, there exists $\widetilde{g} \in \Hsmooth$ such that $\pi (\Lambda) \widetilde{g}$ 
forms a Riesz sequence in $\Hpi$. 
Let $\{r_n\}_{n \in \mathbb{N}} \subset (0,1)$ be such that $r_n \to 1$ as $n \to \infty$, 
and set $\Lambda_n = D_{r_n} (\Lambda)$. By Theorem \ref{thm:riesz_stability}, the system $\pi(\Lambda_n) \widetilde{g}$ 
forms a Riesz sequence in $\Hpi$ for all sufficiently large $n \in \mathbb{N}$, 
and $D^+ (\Lambda_n) > d_{\pi}$. This contradicts Theorem \ref{thm:nonstrict_density}(ii)
and completes the proof. 
\end{proof}

\begin{corollary}
Let $(\pi, \Hpi)$ be a 
projective relative discrete series representation of a homogeneous group $G$. 
Suppose that $g \in \Hsmooth$. 
\begin{enumerate} 
\item Let $\Lambda \subset G$ be relatively separated and relatively dense.
If $\CS$ is a $p$-frame for $\CoLp$ for some $p \in [1,\infty]$, then $D^- (\Lambda) > d_{\pi}$. 
\item Let $\Lambda \subset G$ be separated. If $\CS$ is a $p$-Riesz sequence 
for $\CoLp$ for some $p \in [1,\infty]$, then $D^+ (\Lambda) < d_{\pi}$. 
\end{enumerate}
\end{corollary}

\appendix

\section{Homogeneous quotient groups} \label{sec:homogeneous_quotient}
The purpose of this appendix is to show that the quotient $G/Z(G)$ of a homogeneous group $G$ 
and its center $Z(G)$ can be made 
into a homogeneous group in a canonical fashion. 

\begin{lemma}
\label{lemma_hom_q}
Let $G$ be a homogeneous group. Then the quotient $G/Z(G)$ is also homogeneous.
\end{lemma}
\begin{proof}
Write the family of dilations $\{ D_r^{\mathfrak{g}} \}_{r > 0}$ on the Lie algebra $\mathfrak{g}$ 
as a one-parameter subgroup
\begin{align*}
D_r^{\mathfrak{g}} = \exp_{\mathrm{GL}(\mathfrak{g})}(A \ln(r)) 
= \exp_{\mathrm{GL}(\mathfrak{g})}(t A) =: V(t)
\end{align*}
of $\mathrm{GL}(\mathfrak{g})$ in the parameter $t \in \mathbb{R}$. 
Since each $D_r^{\mathfrak{g}}$ is an automorphism of $\mathfrak{g}$, 
it leaves the center $\mathfrak{z} (\mathfrak{g})$ invariant, that is, $D_r^{\mathfrak{g}} \bigl( \mathfrak{z}(\mathfrak{g}) \bigr) = \mathfrak{z}(\mathfrak{g})$ for all $r > 0$. 
Thus, for fixed $Z \in \mathfrak{z}(\mathfrak{g})$, the map
$t \mapsto V(t)Z$ from $\mathbb{R}^+ $ into $ \mathfrak{z}(\mathfrak{g})$ is a $C^1$-curve.
 Since $\mathfrak{z}(\mathfrak{g})$ is an ideal in  $\mathfrak{g}$, 
 it follows that $\lim_{h \to 0} \frac{1}{h} \bigl( V(h)Z - Z \bigr) = AZ \in \mathfrak{z}(\mathfrak{g})$.
 Thus we have that $A (\mathfrak{z}(\mathfrak{g})) \subseteq \mathfrak{z}(\mathfrak{g})$. 

 We next show that
$\overline{\mathfrak{g}} := \mathfrak{g}/\mathfrak{z}(\mathfrak{g})$ 
admits a family of dilations
\begin{align*}
D_r^{\overline{\mathfrak{g}}} = \exp_{\mathrm{GL}(\overline{\mathfrak{g})}}(\overline{A} \log(r)), \quad r > 0,
\end{align*}
for a diagonalizable matrix $\overline{A}$ with eigenvalues greater $0$. 

Let $\{X_1, ..., X_{\dim (\mathfrak{g})} \}$ be the eigenvectors of $A$. 
Define a linear map $\overline{A} : \overline{\mathfrak{g}} \to \overline{\mathfrak{g}}$ 
by 
\[ \overline{A} (\overline{X}) = A (X + \mathfrak{z}(\mathfrak{g})),\] 
where $X \in \mathfrak{g}$ is such that $\overline{X} = X + \mathfrak{z}(\mathfrak{g})$. 
Note that $\overline{A} : \overline{\mathfrak{g}} \to \overline{\mathfrak{g}}$ is well-defined 
since $A \mathfrak{z}(\mathfrak{g}) \subseteq \mathfrak{z}(\mathfrak{g})$. 
Let $v_j$ be an eigenvalue of $A$ with corresponding eigenvector $X_j$.  
Then $\overline{A} (\overline{X}_j) = v_j \overline{X}_j$.
Thus, if $X_j \notin \mathfrak{z}(\mathfrak{g})$, 
then $\overline{X}_j$ is a non-zero eigenvector of $\overline{A}$ with $v_j > 0$. 
On the other hand, if $X_j \in \mathfrak{z}(\mathfrak{g})$, then $\overline{X}_j = 0$. 
Since $\overline{\mathfrak{g}} = \mathbb{R}$-$\spann \{ \overline{X}_1, ..., \overline{X}_{\dim(\mathfrak{g})} \}$, 
it follows that, by removing finitely many vectors if necessary, 
we obtain an eigenbasis of $\overline{\mathfrak{g}}$. 
\end{proof}

\section{Universality of frames and Riesz sequences}

\subsection{Weighted Schur algebra} \label{sec:schur_algebra}
The following is a Wiener-type lemma for the Schur class on a homogeneous group. 
The result is a special case of an analogous results for matrices 
over discrete metric spaces possessing the polynomial growth property
\cite{sun2007wiener}. 

\begin{theorem} \label{thm:schur_algebra}
Let $G$ be a homogeneous group and let $\Gamma \subseteq G$ 
be a relatively separated set. 
For $\alpha > 0$, let $v_{\alpha} : \Gamma \times \Gamma \to \mathbb{R}^+, \; (\gamma, \gamma') \mapsto (1+|\gamma^{-1} \gamma'|_G )^{\alpha}$. 
Then the Schur class
\[
\mathcal{A}^1_{v_{\alpha}} (\Gamma) 
:= \bigg\{ A \in \mathbb{C}^{\Gamma \times \Gamma} \; : \;  \sup_{\gamma \in \Gamma} \sum_{\gamma' \in \Gamma} v_{\alpha} (\gamma, \gamma') |A_{\gamma, \gamma'}| + \sup_{\gamma' \in \Gamma} \sum_{\gamma \in \Gamma} 
v_{\alpha} (\gamma, \gamma') |A_{\gamma, \gamma'}| < \infty \bigg\}
\]
forms a Banach $*$-algebra. Moreover, it is inverse-closed  and pseudo-inverse
closed in $\mathcal{B}(\ell ^2(\Gamma ))$. 
\end{theorem}
\begin{proof}
The result follows by combining \cite[Theorem 4.1]{sun2007wiener} and \cite[Theorem 5.1]{sun2007wiener} once we verified the standing hypotheses \cite[Section 2]{sun2007wiener} on the index set $\Gamma$ and weight $w_{\alpha}$. 

Consider the metric space $(\Gamma, d_{\Gamma} )$, with $d_{\Gamma} := d_G |_{\Gamma}$ 
being the restriction of the homogeneous metric $d_G$. 
The triple $(G, d_G, \mu_G)$ is a space of homogeneous type 
and the Haar measure $\mu_G$ satisfies the doubling property
\cite[Lemma 3.2.12]{fischer2016quantization}.
Using this, together with the relative separatedness of $\Gamma \subseteq G$, 
it follows by a packing argument that
\[
\mu_c (\Gamma \cap B_r (\gamma) ) \lesssim \rel(\Gamma) \mu_G (B_{r+1} (\gamma) ) 
\leq \rel (\Gamma) 2^Q r^Q \mu_G (B_1 (e)) \lesssim r^Q
\]
for all $\gamma \in \Gamma$ and $r \geq 1$, where $\mu_c$ is the counting measure on $\Gamma$. 
This shows that the triple $(\Gamma, d_{\Gamma}, \mu_c)$ satisfies the so-called polynomial growth property \cite[Section 2.1]{sun2007wiener}. 

The weight $v_{\alpha}$ is admissible \cite[Section 2.2]{sun2007wiener} for all $\alpha > 0$ by \cite[Example A.2]{sun2007wiener}. 
\end{proof}

\subsection{Envelopes in the strong amalgam space}

The following result is essentially \cite[Lemma 6]{romero2012characterization}. 
For completeness, we include the proof. 

\begin{lemma} \label{lem:aux_lemma}
Let $ \Gamma \subset G$ be a  relatively separated set in  a homogeneous group $G$.
 If $F_1, F_2 \in \WRstw(G)$ for $\alpha \geq 0$, then also the function $H : G \to \mathbb{C}$,
\[
H(x)= \sup _{y\in G} \sum_{\gamma \in \Gamma} F_1 (\gamma^{-1} y) F_2 (\gamma^{-1} y x) 
\]
belongs to $\WRstw(G)$.
\end{lemma}
\begin{proof}
Let $x, y \in G$ be fixed.
Then
\begin{align*}
\sum_{\gamma \in \Gamma} F_1 (\gamma^{-1} y) F_2 (\gamma^{-1} yx) 
&\lesssim \sum_{\gamma \in \Gamma} 
\int_G (F_1)_{\sharp} (z^{-1} \gamma^{-1} y) (F_2)_{\sharp} (z^{-1} \gamma^{-1} y x) \mathds{1}_{B_1 (e)} (z) \; d\mu_G (z) \\
&= \int_G (F_1)_{\sharp} (z^{-1}) (F_2)_{\sharp} (z^{-1} x) \sum_{\gamma \in \Gamma} \mathds{1}_{B_1 (e)} (\gamma^{-1} y z) \; d\mu_G (z). 
\end{align*}
Since $\Gamma \subset G$ is relatively separated, 
it follows 
$
\sum_{\gamma \in \Gamma} \mathds{1}_{B_1 (e)} (\gamma^{-1} y z) 
= \sum_{\gamma \in \Gamma} \mathds{1}_{B_1 (\gamma )} ( y z) \lesssim 1,
$
and hence 
\[
H(x) \lesssim \int_G (F_1)_{\sharp} (z^{-1}) (F_2)_{\sharp} (z^{-1} x) \; d\mu_G (z). 
\]
Direct calculations next entail that
\[
H_{\sharp} (x) \lesssim \sup_{v \in B_1 (e)} \int_G (F_1)_{\sharp} (z^{-1} v^{-1}) (F_2)_{\sharp} (z^{-1} x ) \; d\mu_G (z) 
\leq \int_{G} ((F_1)_{\sharp})^{\sharp} (z^{-1}) (F_2)_{\sharp} (z^{-1} x) \; d\mu_G (z)
\]
and
\[
(H_{\sharp})^{\sharp} (x) \lesssim \int_G ((F_1)_{\sharp})^{\sharp} (z^{-1}) ((F_2)_{\sharp})^{\sharp} (z^{-1} x) \; d\mu_G (z).
\]
Using the submultiplicativity $w_{\alpha} (x) \leq w_{\alpha} (z) w_{\alpha} (z^{-1} x)$ 
and Fubini's theorem gives
\begin{align*}
\| H \|_{\WRst} 
 &\lesssim \int_G w_{\alpha} (x) \int_G ((F_1)_{\sharp})^{\sharp} (z^{-1}) ((F_2)_{\sharp})^{\sharp} (z^{-1} x) \; d\mu_G (z) d\mu_G (x) \\ 
&\leq  \int_G  ((F_1)_{\sharp})^{\sharp} (z^{-1}) w_{\alpha} (z) \int_G ((F_2)_{\sharp})^{\sharp} (z^{-1} x) w_{\alpha} (z^{-1} x)\; d\mu_G (x) d\mu_G (z) \\
&= \| ((F_2)_{\sharp})^{\sharp} \|_{L^1_{w_{\alpha}}} \int_G ((F_1)_{\sharp})^{\sharp} (z) w_{\alpha} (z^{-1})  \; d\mu_G (z) \\
&= \| F_1 \|_{\WRstw} \|F_2 \|_{\WRstw}, 
\end{align*}
as desired.
\end{proof}

\subsection{$\ell^p$-stability of matrices}
The following result elaborates
on Sj\"{o}strand's Wiener-type lemma \cite{sjostrand1995wiener}. 
Variations of this result have been derived multiple times 
\cite{aldroubi2008slanted, tessera2010left, shin2009stability, shin2019polynomial}, 
but none of these seem to be directly applicable for a version required for our purposes. 
We present  a version valid on homogeneous groups. 
The proof structure follows \cite[Proposition A.1]{groechenig2015deformation} 
very closely. However, in contrast to \cite[Proposition A.1]{groechenig2015deformation}, 
in the non-commutative case a strictly polynomial weight is assumed. 

\begin{proposition} \label{prop:sjostrand-wienertype}
Let $\Lambda, \Gamma \subset G$ be relatively separated subsets 
in a homogeneous group $G$, with homogeneous dimension $Q$.  
Let $\alpha \geq Q$. Suppose that
 $A \in \mathbb{C}^{\Lambda \times \Gamma}$ is a matrix 
for which there exists  $\Theta \in \WRstw(G)$ such that
\begin{align} \label{eq:wiener_envelope}
|A_{\lambda, \gamma} | \leq \Theta (\lambda^{-1} \gamma), 
\quad \lambda \in \Lambda, \gamma \in \Gamma. 
\end{align}
Moreover, suppose there exists  $p \in [1,\infty]$ and $C = C(p) > 0$ such that, 
for all $c \in \ell^p (\Gamma)$, 
\begin{align} \label{eq:sjostrand-wiener-assumption}
C \|c \|_{\ell^p} \leq \| A c \|_{\ell^p}. 
\end{align}
Then there exists  $C' > 0$ such that, for all $q \in [1,\infty]$ and all $c \in \ell^q (\Gamma)$,
\begin{align*}
C' \|c \|_{\ell^q} \leq \| A c \|_{\ell^q}. 
\end{align*}
\end{proposition}
\begin{proof}
We prove the result in several steps. 

\textbf{Step 1.} (Partition of unity). 
Let $\{ B_{1/2} (x_k) \}_{k \in \mathbb{N}}$ 
be a maximal family of disjoint balls, 
with centers $X = \{x_k \}_{k \in \mathbb{N}} \subseteq G$. 
Then, by maximality, the balls $\{B_1 (x_k) \}_{k \in \mathbb{N}}$ form a cover of $G$. 
Moreover, for arbitrary  $C_b \geq 1$, no point $x \in G$ belongs to more than 
$\lceil (4 C_b)^Q \rceil$ many of the balls $\{B_{C_b} (x_n)\}_{n \in \mathbb{N}}$
- see for example \cite[Lemma 5.7.5]{fischer2016quantization}.
Associated to $\{ x_k \}_{k \in \mathbb{N}}$, there exists a
partition of unity $\{ \psi_k \}_{k \in \mathbb{N}}$ of functions $\psi_k \in C_c^{\infty} (G)$ 
satisfying $\supp \psi_k \subset B_2 (x_k)$,  $\psi_k (G) \subset [0,1]$ 
and $ \sum_{k \in \mathbb{N}} \psi_k (\cdot) = 1$. 
For fixed $\varepsilon \in (0,1]$, let $\psike := \psi_k \circ D_{\varepsilon}$. 
Then $\supp \psike \subseteq D_{1/\varepsilon} (B_2 (x_k)) = B_{2/\varepsilon} (D_{1/\varepsilon} (x_k))$ 
and $\Psi^{\varepsilon} := \sum_{k \in \mathbb{N}} (\psike)^2 \asymp 1$, 
with constants independent of $\varepsilon$. 
Moreover, by an application of the mean value theorem \cite[Proposition 3.1.46]{fischer2016quantization}, 
it follows that $|\psike (x) - \psike (y) | \lesssim \varepsilon^N |x^{-1} y|^N_G$ for some $N \in \mathbb{N}$ 
and all $x,y \in G$. Combined with $\psike \leq 1$, this gives
\begin{align} \label{eq:lipschitz_bound}
|\psike (x) - \psike (y) | \lesssim \min \{1, \varepsilon^N |x^{-1} y|^N_G \}
\end{align}
for all $x,y \in G$. 

\textbf{Step 2.} (Norm equivalence).
Let $p \in [1,\infty]$. For fixed $\varepsilon \in (0, 1]$ and $k \in \mathbb{N}$, 
define the multiplication by $\psike|_{\Gamma} \in \ell^{\infty} (\Gamma)$ as the operator
$\psike : \ell^p(\Gamma) \to \ell^p (\Gamma)$ given by $(\psike c)_{\gamma} = \psike(\gamma) c_{\gamma}$.
We show that, for every $q \in [1,\infty]$ and $c \in \ell^q (\Gamma)$, 
\begin{align} \label{eq:wiener-claim2}
\big\| \big( \| \psike c \|_{\ell^p} \big)_{k \in \mathbb{N}} \big\|_{\ell^q}
\asymp \| c \|_{\ell^q},
\end{align} 
with constants independent of $p,q \in [1,\infty]$. 
Note that, for every fixed $\varepsilon > 0$, we have
 $N_{\varepsilon} := \sup_{k \in \mathbb{N}} \# \supp (\psike|_{\Gamma}) < \infty$ 
 since $\Gamma \subset G$ is relatively separated. 
From this, it follows that, for every $q \in [1,\infty]$ and $c \in \ell^{\infty} (\Gamma)$, 
we have $\| \psike c \|_{\ell^p} \leq \| \psike c \|_{\ell^1} \leq N_{\varepsilon} \| \psike c \|_{\infty} \leq N_{\varepsilon} \| \psike c \|_{\ell^q}$ and
$\|\psike c\|_{\ell^q} \leq N_{\varepsilon} \| \psike c \|_{\ell^p}$. 
Consequently, for all $c \in \ell^q (\Gamma)$,
\begin{align} \label{eq:step1-claim2}
\big\| \big( \| \psike c \|_{\ell^p} \big)_{k \in \mathbb{N}} \big\|_{\ell^q}
\asymp 
\big\| \big( \| \psike c \|_{\ell^q} \big)_{k \in \mathbb{N}} \big\|_{\ell^q},
\end{align}
with constants independent of $p,q \in [1,\infty]$. 
Since $X = \{x_k\}_{k \in \mathbb{N}}$ is relatively separated, it follows
$
\eta_{\varepsilon} := \sup_{x \in G} \# \{ k \in \mathbb{N}  :
 \psike (x) \neq 0 \} \leq  \sup_{x \in G} \# \{ k \in \mathbb{N}  :
 D_{\varepsilon} (x) \in B_2(x_k)  \} 
 = \sup_{x \in G} \# \{ k \in \mathbb{N}  :
x_k \in B_2( D_{\varepsilon} (x))  \} \lesssim \rel(X).
$ 
Therefore $\eta := \sup_{\varepsilon \in (0,1]} \eta_{\varepsilon} < \infty$.
Hence 
$
1 = \sum_{k \in \mathbb{N}} \psike (x) 
\leq \eta \; \sup_{k \in \mathbb{N}} \psike (x)
$
for all $x \in G$, and 
\[
\frac{1}{\eta} \leq \sup_{k \in \mathbb{N}} \psike (x) 
\leq \bigg( \sum_{k \in \mathbb{N}} (\psike(x))^q \bigg)^{\frac{1}{q}} 
\leq \sum_{k \in \mathbb{N}} \psike (x) = 1. 
\]
Therefore, if $q \in [1,\infty)$ and $c \in \ell^q (\Gamma)$, then
\[
\frac{1}{\eta^q} \sum_{\gamma \in \Gamma} |c_{\gamma} |^q 
\leq \sum_{\gamma \in \Gamma} \sum_{k \in \mathbb{N}} (\psike(\gamma))^q |c_{\gamma}|^q 
\leq \sum_{\gamma \in \Gamma} |c_{\gamma} |^q.
\]
Similarly, if $q = \infty$ and $c \in \ell^{\infty} (\Gamma)$,
then 
\[
\eta^{-1} \sup_{\gamma \in \Gamma} |c_{\gamma}| 
\leq \sup_{\gamma \in \Gamma} \sup_{k \in \mathbb{N}} \psike (\gamma) |c_{\gamma}|
\leq \sup_{\gamma \in \Gamma} |c_{\gamma}|.
\]
Thus, for any $q \in [1,\infty]$, 
\begin{align} \label{eq:norm-equivalence2}
\big\| \big( \|\psike c \|_{\ell^q} \big)_{k \in \mathbb{N}} \big\|_{\ell^q} 
\asymp \| c \|_{\ell^q},
\end{align}
with constants independent of $q \in [1,\infty]$. 
Combining \eqref{eq:step1-claim2} and \eqref{eq:norm-equivalence2} yields \eqref{eq:wiener-claim2}. 

\textbf{Step 3.} (Auxiliary matrix $V^{\varepsilon}$).
Let $A \in \mathbb{C}^{\Lambda \times \Gamma}$. 
Consider $[A, \psike] := A \psike - \psike A$.  
Assume, without loss of generality, 
that \eqref{eq:sjostrand-wiener-assumption} holds with $C = 1$. We set
$K := \max_x (\Psi^{\varepsilon} (x))^{-1}$ and $V_{j,k}^{\varepsilon}
:= \| [A, \psike]\psije\|_{\schur(\Gamma \to \Lambda)}$ for  $j, k \in \mathbb{N}$,  
where 
\[
\| B \|_{\schur(J \to I)} := \max \bigg\{ \sup_{i \in I} \sum_{j \in J} |B_{j,i}|, \sup_{j \in J} \sum_{i \in I} |B_{j,i}| \bigg\}
\]
 denotes the Schur norm of a matrix $B \in \mathbb{C}^{I \times J}$. 
Then a direct calculation entails 
\begin{align*} 
\| \psike c \|_{\ell^p} &\leq \| \psike A c \|_{\ell^p } + \sum_{j \in \mathbb{N}} \| [A, \psike] \psije (\Psi^{\varepsilon})^{-1} \psije c \|_{\ell^p} \\
&\leq \| \psike A c \|_{\ell^p} + K \sum_{j \in \mathbb{N}} V_{j,
  k}^{\varepsilon} \| \psije c \|_{\ell^p}. \numberthis \label{eq:multiplier_estimate}
\end{align*} 

\textbf{Step 4.} (Uniform convergence of the entries $V^{\varepsilon}_{j,k}$).
We claim that
\begin{align} \label{eq:convergence_entries}
\sup_{j,k} V_{j,k}^{\varepsilon} \to 0 \hspace{5pt} \mbox{ as } \hspace{5pt} \varepsilon \to 0^+.
\end{align}
For this, note that 
$
([A, \psike] \psije)_{\lambda, \gamma} 
= - A_{\lambda, \gamma} \psije (\gamma) 
\big( \psike(\lambda) - \psike(\gamma) \big). 
$
Combining \eqref{eq:wiener_envelope} and \eqref{eq:lipschitz_bound} yields 
\[
\big| ([A, \psike] \psije)_{\lambda, \gamma} \big| 
\lesssim \Theta(\lambda^{-1} \gamma) \min \big\{1, \varepsilon^N |\lambda^{-1} \gamma|_G^N \big\}.
\]
Set $\Theta^{\varepsilon} (x) := \Theta(x) \min\{1, \varepsilon^N |x|_G^N \}$. 
Using the estimate $\sum_{\gamma \in \Gamma} |f (\gamma)| \lesssim \rel(\Gamma) \| f \|_{\WLC}$ 
then gives
\[
V_{j,k}^{\varepsilon} \lesssim \max \{ \rel (\Gamma), \rel (\Lambda) \}. 
\| \Theta^{\varepsilon} \|_{\WRst},
\]
An application of Lebesgue's dominated convergence theorem therefore yields \eqref{eq:convergence_entries}. 

\textbf{Step 5.} (Refined estimates of the entries $V^{\varepsilon}_{j,k}$).
For estimating  $V_{j,k}^{\varepsilon}$, fix $j, k \in \mathbb{N}$. 
Note that $\psije(\gamma) \psike(\gamma) \neq 0$ only if $D_\epsilon (\gamma )\in B_2 (x_j) \cap B_2 (x_k)$.
Thus, if $|x_k^{-1} x_j |_G > 4$, then $ B_2 (x_j) \cap B_2 (x_k) = \emptyset$, and
the entries of $[A, \psike]\psije$ simplify to 
\[
([A, \psike]\psije)_{\lambda, \gamma} = -A_{\lambda, \gamma} \psije(\gamma) \psike(\lambda).
\]
Together with the envelope assumption \eqref{eq:wiener_envelope}, this gives
\[
\sup_{\lambda \in \Lambda} \sum_{\gamma \in \Gamma} | ([A, \psike]\psije)_{\lambda, \gamma} |
\leq \sup_{\lambda \in \Lambda} \sum_{\gamma \in \Gamma} \Theta (\lambda^{-1} \gamma) \psije(\gamma) \psike(\lambda) 
\]
for all $j \in \mathbb{N}$ satisfying $|x_k^{-1} x_j|_G > 4$. Similarly, it follows that
\[
\sup_{\gamma \in \Gamma} \sum_{\lambda \in \Lambda} | ([A, \psike]\psije)_{\lambda, \gamma} |
\leq \sup_{\gamma \in \Gamma} \sum_{\gamma \in \Gamma} \Theta^{\vee} (\gamma^{-1} \lambda) \psije(\gamma) \psike(\lambda), 
\]
yielding the desired estimates for $V_{j,k}^{\varepsilon}$. 

\textbf{Step 6.} (Schur norm of $V^{\varepsilon}$.)
In this step, we will show that
\begin{align} \label{eq:two_limits}
 \sup_{k \in \mathbb{N}} \sum_{j \in \mathbb{N}} V_{j,k}^{\varepsilon} \to 0
 \quad \text{and} \quad
\sup_{j \in \mathbb{N}} \sum_{k \in \mathbb{N}} V_{j,k}^{\varepsilon}  \to 0
\end{align}
as $\varepsilon \to 0^+$, yielding
 that $\| V^{\varepsilon} \|_{\schur(\mathbb{N} \to \mathbb{N})} \to 0$ 
as $\varepsilon \to 0^+$. 
We only show the first limit in \eqref{eq:two_limits}; 
the second limit follows analogously by interchanging the role of $j,k \in \mathbb{N}$. 

Fix $k \in \mathbb{N}$. Then 
\begin{align} \label{eq:split_norm}
\sum_{j \in \mathbb{N}} V_{j,k}^{\epsilon} &= \sum_{j \in \mathbb{N} \; : \;|x_k^{-1} x_j|_G \leq 5}  V_{j,k}^{\epsilon} +  \sum_{j \in \mathbb{N} \; : \; |x_k^{-1} x_j |_G > 5} V_{j,k}^{\epsilon} . 
\end{align}
The first series in the right-hand side 
of \eqref{eq:split_norm} can be estimated by
\[
\sum_{j \in \mathbb{N} \; : \;|x_k^{-1} x_j|_G \leq 5}  V_{j,k}^{\epsilon} 
\leq \# \{ j \in \mathbb{N} \; : \; |x_k^{-1} x_j |_G \leq 5 \} \; \sup_{j \in \mathbb{N}} V_{j,k}^{\epsilon} 
\lesssim \sup_{j \in \mathbb{N}} V_{j,k}^{\epsilon},
\]
with a constant independent of $k$. 
Thus $\sum_{j \in \mathbb{N} \; : \;|x_k^{-1} x_j|_G \leq 5}  V_{j,k}^{\epsilon}  \to 0$ as $\epsilon \to 0^+$ by \eqref{eq:convergence_entries}.  

For fixed $j, k \in \mathbb{N}$, 
choose $\lambda_{j,k} \in \Lambda \cap \supp \psike $ such that
\begin{align*}
\sup_{\lambda \in \Lambda} \sum_{\gamma \in \Gamma} \Theta (\lambda^{-1} \gamma) \psije(\gamma) \psike(\lambda)
&= \sum_{\gamma \in \Gamma} \Theta (\lambda_{j,k}^{-1} \gamma) \psije(\gamma) \psike(\lambda_{j,k}) \\
&= \sum_{\lambda \in \Lambda} \delta_{\lambda_{j,k}} (\lambda) 
\sum_{\gamma \in \Gamma} \Theta (\lambda^{-1} \gamma) \psije(\gamma) \psike(\lambda). 
\end{align*}
Hence, defining $
T := \sum_{j \in \mathbb{N} \; : \; |x_k^{-1} x_j |_G > 5} \sup_{\lambda \in \Lambda} \sum_{\gamma \in \Gamma} \Theta (\lambda^{-1} \gamma) \psije(\gamma) \psike(\lambda) 
$ gives
\begin{align*}
T = \sum_{j \in \mathbb{N} \; : \; |x_k^{-1} x_j |_G > 5} \sum_{\lambda \in \Lambda} \sum_{\gamma \in \Gamma} \Theta (\lambda^{-1} \gamma) \psije(\gamma) \psike(\lambda) \delta_{\lambda_{j,k}} (\lambda).
\end{align*}
Interchanging sums and using that
 $\sum_j \delta_{\lambda_{j,k}} (\lambda) \psije(\gamma) \leq 1$ for $\lambda \in \Lambda, \;\gamma \in \Gamma$
 yields
 \begin{align*}
T \leq  \sum_{\lambda \in \Lambda} \sum_{\gamma \in \Gamma} \Theta (\lambda^{-1} \gamma) \psike(\lambda) 
\lesssim \rel(\Gamma) \| \Theta \|_{W(L^{\infty}, L^1)} \sum_{\lambda \in \Lambda} \psike (\lambda) 
\lesssim \varepsilon^{-Q} \| \Theta \|_{W(L^{\infty}, L^1)}, \numberthis \label{eq:amalgam_epsilon}
 \end{align*}
where the last step follows from $\sum_{\lambda \in \Lambda} \psike(\lambda) \leq \# \big( \Lambda \cap B_{2/\varepsilon} (D_{1/\varepsilon} (x_k)) \big) \lesssim \varepsilon^{-Q} \rel(\Lambda)$.  

For $\varepsilon \in (0,1]$, write $\Theta = \Theta_1 + \Theta_2$, where $\Theta_1 := \Theta \cdot  \mathds{1}_{B_{1/\varepsilon} (e)}$ 
and $\Theta_2 := \Theta \cdot (\mathds{1}_G - \mathds{1}_{B_{1/\varepsilon} (e)})$. 
If $\psije (\gamma) \neq 0$ and $\psike (\lambda) \neq 0$, then $|D_{\varepsilon} (\lambda^{-1} \gamma)|_G \geq |x_k^{-1} x_j|_G - 4$. Thus, if $|x_k^{-1} x_j|_G > 5$, then $|\lambda^{-1} \gamma|_G > 1/\varepsilon$, 
yielding that $\Theta_1 (\lambda^{-1} \gamma) = 0$. Combining this with  \eqref{eq:amalgam_epsilon} gives
\begin{align*}
T
\lesssim  \varepsilon^{-Q} \| \Theta \cdot (\mathds{1}_G - \mathds{1}_{B_{1/\varepsilon} (e)} ) \|_{W(L^{\infty}, L^1)}. 
 \numberthis \label{eq:amalgam_epsilon2}
\end{align*}
For a sequence $\{x_n \}_{n \in \mathbb{N}}$ of points $x_n \in G$ as in Step 1, the norm
\[
\| f \|_{W(L^{\infty}, \ell^1_{w_{\alpha}})} := \sum_{n \in \mathbb{N}} 
\; (1+|x_n|_G)^{\alpha} \| f \|_{L^{\infty} (B_1 (x_n))}
\]
defines an equivalent norm on $W(L^{\infty}, L^1_{w_{\alpha}}) (G)$. 
Therefore
\begin{align*}
\| \Theta \cdot (\mathds{1}_G - \mathds{1}_{B_{1 / \varepsilon} (e)} ) \|_{W(L^{\infty}, L^1)} 
&\leq \sum_{|x_n|> \frac{1}{\varepsilon} - 1} \| \Theta \|_{L^{\infty} (B_1 (x_n))} (1+ |x_n|_G)^{\alpha} (1+|x_n|_G)^{-\alpha} \\
&\leq \varepsilon^{\alpha} \sum_{|x_n|> \frac{1}{\varepsilon} - 1} \| \Theta \|_{L^{\infty} (B_1 (x_n))} (1+ |x_n|_G)^{\alpha}. 
\end{align*}
Combining this with the estimate \eqref{eq:amalgam_epsilon2} thus gives
\begin{align*}
T &= \sum_{j \in \mathbb{N} \; : \; |x_k^{-1} x_j |_G > 5} \sup_{\lambda \in \Lambda} \sum_{\gamma \in \Gamma} \Theta (\lambda^{-1} \gamma) \psije(\gamma) \psike(\lambda) \\
&\lesssim \varepsilon^{\alpha - Q} \sum_{|x_n|> \frac{1}{\varepsilon} - 1} \| \Theta \|_{L^{\infty} (B_1 (x_n))} (1+ |x_n|_G)^{\alpha},
\end{align*}
with the right-hand side tending to $0$ as $\varepsilon \to 0^+$ 
since $\Theta \in \WRstw (G)$ and $\alpha \geq Q$.
By interchanging the role of $\Gamma$ and $\Lambda$, it follows similarly that
\[
\sum_{j \in \mathbb{N} \; : \; |x_k^{-1} x_j |_G > 5} \sup_{\lambda \in \Lambda} \sum_{\gamma \in \Gamma} \Theta^{\vee} (\gamma^{-1} \lambda) \psije(\gamma) \psike(\lambda) 
\to 0 \quad \text{as} \quad \varepsilon \to 0^+. 
\]
Hence, by Step 5 and combining both limits gives
$
\sum_{ j \in \mathbb{N} \; : \; |x_k^{-1} x_j |_G > 5} V_{j,k}^{\epsilon} \to 0
$
as $\varepsilon \to 0^+$. 
This proves the first limit in \eqref{eq:two_limits}. 

\textbf{Step 7.} (Conclusion.) 
By Step 6, there exists an $\varepsilon > 0$ such that, for all $a \in \ell^q (\mathbb{N})$, 
we have
$
\| V^{\varepsilon} a \|_{\ell^q} \leq (2K)^{-1} \| a \|_{\ell^q} 
$
uniformly for all $q \in [1,\infty]$. 
Applying this in \eqref{eq:multiplier_estimate} yields
\begin{align} \label{eq:norm-equivalence1}
\frac{1}{2} \, \big\| \big( \| \psike c \|_{\ell^p} \big)_{k \in \mathbb{N}} \big\|_{\ell^q}
\leq \big\| \big( \| \psike Ac \|_{\ell^p} \big)_{k \in \mathbb{N}} \big\|_{\ell^q}. 
\end{align}
Combining the norm equivalences \eqref{eq:norm-equivalence1} 
and \eqref{eq:wiener-claim2} completes the proof. 
\end{proof}

\subsection{Existence of a localized reference frame}
In order to apply results on the spectrum of matrices to problems in frame theory we need to know that there exist an adequate
(reference) frame. The following proposition serves that purpose, 
see also the first sections of \cite{MR2765595} and \cite[Section
7]{groechenig2018strict}. The result improves the existence results in
\cite{feichtinger1989banach1,groechenig1991describing} 
by adding fine information about the canonical dual frame. 

\begin{proposition} \label{prop:canonical_dual}
Let $(\pi, \Hpi)$ be a 
projective relative discrete series representation of a homogeneous group $G$. 
Suppose that $h \in \Hsmooth$.
Then there exists a relatively separated and relatively dense set
$\Gamma \subset G$ such that $\pi (\Gamma) h$ forms a $p$-frame for 
$\CoLp$ for all $p \in [1, \infty]$. The canonical dual frame 
$\{\widetilde{h}_{\gamma} \}_{\gamma \in \Gamma}$  
of $\pi (\Gamma)h$ in $\Hpi$ is $\WRstw$-\emph{localized} for every $\alpha \geq 0$ in the sense that 
there exists a $\widetilde{\Theta} \in \WRstw$ such that
\begin{align} \label{eq:canonical_dual_localized}
|V_h \widetilde{h}_{\gamma} (x)| \leq  \widetilde{ \Theta}(\gamma^{-1} x)
\end{align}
for all $\gamma \in \Gamma$ and $x \in G$.  
As a consequence, any $f \in \CoLp$ admits an expansion
\[
f = \sum_{\gamma \in \Gamma} \langle f, \pi (\gamma) h \rangle \widetilde{h}_{\gamma} 
= \sum_{\gamma \in \Gamma} \langle f, \widetilde{h}_{\gamma} \rangle \pi(\gamma) h
\]
with norm convergence if $p \in [1,\infty)$ and weak$^*$-convergence in $\Honedual$, otherwise. 
Moreover, 
\begin{align} \label{eq:norm_equiv_canonical}
\| f \|_{\CoLp} 
\asymp \| \{ \langle f, \widetilde{h}_{\gamma} \rangle \}_{\gamma \in \Gamma} \|_{\ell^p} 
\end{align}
 for all $f \in \CoLp$. 
\end{proposition}
\begin{proof}
Let $\alpha \geq 0$. Let $h \in \Hsmooth \subset \Bpi$.
By the main results of \cite{feichtinger1989banach1,groechenig1991describing}
there exists a relatively separated and relatively dense set
$\Gamma \subset G$ such that $\pi(\Gamma) h$ forms a frame for $\Hpi$,
 see also \cite{christensen1996atomic}. Let $\{\widetilde{h}_{\gamma} \}_{\gamma \in \Gamma}$ be the canonical 
dual frame of $\pi (\Gamma)h$ in $\Hpi$.
To show the localization estimate \eqref{eq:canonical_dual_localized}, write
\[
\widetilde{h}_{\gamma} 
= \sum_{\gamma' \in \Gamma} \langle \widetilde{h}_{\gamma}, \widetilde{h}_{\gamma'} \rangle \pi(\gamma') h
\]
and let $\widetilde{\mathcal{G}} \in \mathbb{C}^{\Gamma \times \Gamma}$ be defined by
$\widetilde{\mathcal{G}}_{\gamma, \gamma'} = \langle \widetilde{h}_{\gamma'}, \widetilde{h}_{\gamma} \rangle$. 
Then $\widetilde{\mathcal{G}} = \mathcal{G}^{\dagger}$,
where $\mathcal{G}^{\dagger}$ denotes the pseudo-inverse of the Gramian matrix
$\mathcal{G} \in \mathbb{C}^{\Gamma \times \Gamma}$ of $\pi(\Gamma) h$, 
defined by $\mathcal{G}_{\gamma, \gamma'} = \langle \pi(\gamma') h, \pi(\gamma) h \rangle$. 
Since $V_h  h \in \mathcal{S} (G)$,
 for every $N \in \mathbb{N}$, there exists a $C = C(N) > 0$ such that
\[
|\langle \pi(\gamma) h , \pi (\gamma') h \rangle| = |V_h h (\gamma^{-1} \gamma')| \leq C (1 + |\gamma^{-1} \gamma'|_G )^{-N}
\]
for all $\gamma, \gamma' \in \Gamma$. 
By choosing $N \in \mathbb{N}$ sufficiently large, it follows that
 $\mathcal{G} \in \mathcal{A}^1_{v_{s}} (\Gamma)$ for every $s > 0$, 
where $\mathcal{A}^1_{v_{s}} (\Gamma)$ denotes the weighted Schur algebra over $\Gamma$ 
defined in Appendix \ref{sec:schur_algebra}. 
By Theorem \ref{thm:schur_algebra}, it follows that also $\widetilde{\mathcal{G}} = \mathcal{G}^{\dagger} \in \mathcal{A}^1_{v_{s}} (\Gamma)$. 
As a consequence, this yields in particular that
$|\widetilde{\mathcal{G}}_{\gamma, \gamma'}| \leq v_{s}(\gamma, \gamma')^{-1} = (1 + |\gamma^{-1} \gamma'|_G)^{-s}$ for $\gamma, \gamma' \in \Gamma$. 
Consider $\Theta(x) = w_{-s} (x) = (1+|x|_G)^{-s}$ for $x \in G$. 
Then $(\Theta_{\sharp})^{\sharp} (x) \lesssim (1+|x|_G)^{-s}$ for all $x \in G$, 
with an implicit constant independent of $x$.
Thus, choosing $s > 0$ sufficiently large,
 it follows that $(\Theta_{\sharp})^{\sharp} \in L^1_{w_{\alpha}} (G)$ 
and $\Theta \in \WRstw(G)$. 
Consequently, for any $\gamma \in \Gamma$ and $x \in G$, 
\begin{align*}
|V_h \widetilde{h}_{\gamma} (x)| 
&\leq \sum_{\gamma' \in \Gamma} |\langle \widetilde{h}_{\gamma}, \widetilde{h}_{\gamma'} \rangle || V_h \pi (\gamma') h (x) | \leq  \sum_{\gamma' \in \Gamma} \Theta(\gamma'^{-1} \gamma) |V_h h| (\gamma'^{-1} \gamma  (\gamma^{-1} x)) \\
&\leq \sup_{\gamma \in \Gamma} \sum_{\gamma' \in \Gamma} \Theta(\gamma'^{-1} \gamma) |V_h h| (\gamma'^{-1} \gamma  (\gamma^{-1} x)) 
=: \widetilde{\Theta} (\gamma^{-1} x)
\end{align*}
where $\widetilde{\Theta} \in \WRstw (G)$ by Lemma \ref{lem:aux_lemma}. 
This shows \eqref{eq:canonical_dual_localized}. 

Lastly, by \cite[Lemma 3.4]{groechenig2009molecules} the operators 
$
C_{\widetilde{H}} : \CoLp \to \ell^p (\Gamma), \; 
f \mapsto \{ \langle f, \widetilde{h}_{\gamma} \rangle \}_{\gamma \in \Gamma}
$ 
and $D_{\widetilde{H}} : \ell^p (\Gamma) \to \CoLp, \; 
(c_{\gamma})_{\gamma \in \Gamma} \mapsto \sum_{\gamma \in \Gamma} c_{\gamma} \widetilde{h}_{\gamma}$ 
are well-defined and bounded, 
with $D_{\widetilde{H}}$ satisfying the desired convergence properties.
As a consequence, the identity
$
f = C_{h, \Gamma} D_{\widetilde{H}}f = C_{\widetilde{H}} D_{h, \Gamma}f
$
holds for all $f \in \CoLp$. 
The norm equivalence \eqref{eq:norm_equiv_canonical} follows from
\begin{align*}
\|f \|_{\CoLp} &= \| D_{h, \Gamma} C_{\widetilde{H}} f \|_{\CoLp}  
\leq  \| D_{h,\Gamma} \|_{op} \| C_{\widetilde{H}} f \|_{\ell^p (\Gamma)} \\
&\leq
 \| D_{h, \Gamma} \|_{op} \| C_{\widetilde{H}} \|_{op} \|f \|_{\CoLp}.
\end{align*}
Similarly, it follows that $\|f \|_{\CoLp} \asymp \|C_{h,\Gamma} \|_{\ell^p}$,
showing that $\pi(\Gamma)h$ forms a $p$-frame for $\CoLp$ for all $p \in [1,\infty]$. 
This completes the proof. 
\end{proof}

\subsection{Boundedness below on a subspace}

The reference frame provided by Proposition \ref{prop:canonical_dual} allows us to reformulate certain properties of a general frame in terms of corresponding properties of its Gram matrix. The redundancy of the reference frame poses certain obstacles that can be circumvented with 
an extension of Proposition \ref{prop:sjostrand-wienertype}, as done in \cite[Section 7]{groechenig2018strict}. We quote
\cite[Theorem 7.1]{groechenig2018strict} here 
and repeat its proof in the context of a homogeneous group.

\begin{theorem} \label{thm:wiener-type-extension}
Let $G$ be a homogeneous group with homogeneous dimension $Q$.
Let $\Lambda, \Gamma \subseteq G$ be relatively separated sets, 
and let $P : \ell^2 (\Gamma) \to \ell^2 (\Gamma)$
and $A : \ell^2 (\Gamma) \to \ell^2 (\Lambda)$ be bounded linear operators. 
Suppose that $P$ is idempotent, i.e., $P^2 = P$, and that there exist 
$\Theta_1, \Theta_2 \in \WRstw (G)$ for $\alpha \geq Q+1$ such that
\begin{align*}
|A_{\lambda, \gamma}| \leq \Theta_1 (\lambda^{-1} \gamma), \quad \lambda \in \Lambda, \gamma \in \Gamma
\end{align*}
and 
\begin{align*}
|P_{\gamma,\gamma'}| \leq \Theta_2 (\gamma^{-1} \gamma'), \quad \gamma, \gamma' \in \Gamma. 
\end{align*}
If there is some $p \in [1,\infty]$ and $C = C(p) > 0$ such that, for all $c \in \ell^p (\Gamma)$, 
\begin{align} \label{eq:wiener-type_somep}
C \|Pc \|_{\ell^p} \leq \|A P c \|_{\ell^p},
\end{align}
then there exists a $C' > 0$, independent of $q$, such that for all $q \in [1,\infty]$ and all $c \in \ell^q (\Gamma)$,
\[
C' \|Pc \|_{\ell^q} \leq \|A P c \|_{\ell^q}.
\]
\end{theorem}
\begin{proof}
Consider $\widetilde{A} : \ell^2 (\Gamma) \to \ell^2 (\Lambda) \oplus \ell^2 (\Gamma),
\;  c \mapsto ((AP)c, (I-P)c).$ Then the operators $AP$ and $I-P$ 
defining $\widetilde{A}$ satisfy envelope conditions
\begin{align} \label{eq:AP_localization}
|(AP)_{\lambda, \gamma}| \leq \Theta'_1 (\lambda^{-1} \gamma), \quad  \lambda \in \Lambda, \gamma \in \Gamma
\end{align}
and 
\begin{align} \label{eq:IP_localization}
|(I-P)_{\gamma, \gamma'}| \leq \Theta'_2 (\gamma^{-1} \gamma'), \quad \gamma, \gamma' \in \Gamma
\end{align}
for some $\Theta_1', \Theta_2' \in \WRstw(G)$. 
The envelope condition \eqref{eq:AP_localization} follows by Lemma \ref{lem:aux_lemma}, 
whereas \eqref{eq:IP_localization} is immediate. 

We will show that \eqref{eq:wiener-type_somep} implies that 
$\widetilde{A}$ is $q$-bounded below for all $q \in [1,\infty]$. 
For this, we construct an auxillary operator $B$ to which Proposition 
\ref{prop:sjostrand-wienertype} applies. 
For this, note that $G \times \mathbb{R}$ is a homogeneous of homogeneous dimension $Q+1$,
 when equipped 
with the canonical dilations. 
Define the relatively separated sets
$\Lambda^* := \Lambda \times \{0\}$ and
$\Gamma^* := \Gamma \times \{1\}$ in 
$G \times \mathbb{R}$, 
and set $\Omega^* := \Lambda^* \cup \Gamma^*$. 
Then the operator $\widetilde{A}$ can be identified with the operator 
$B : \ell^2 (\Gamma^*) \to \ell^2 (\Omega^*)$ with entries defined by
$B_{(\lambda, 0), (\gamma, 1)} := (AP)_{\lambda, \gamma}$ and 
$B_{(\gamma, 1), (\gamma', 1)} := (I-P)_{\gamma, \gamma'}$. 
To see that $B$ possesses an envelope, take a bump function $\eta \in C_c^{\infty} (\mathbb{R})$ 
with $\supp \eta \subseteq [-2,2]$ and $\eta \equiv 1$ on $[-1,1]$, and define $\Theta' = \Theta_1' + \Theta_2'$ and $\widetilde{\Theta} (x,t) := \Theta'(x) \eta(t)$. Then
$\widetilde{\Theta} \in W(L^{\infty}, L^1_{\widetilde{w_{\alpha}}})(G \times \mathbb{R})$,
with $\widetilde{w_{\alpha}} : G \times \mathbb{R} \to \mathbb{R}^+, \; (1+|x|_G + |t|)^{\alpha}$.
The estimates \eqref{eq:AP_localization} and \eqref{eq:IP_localization} entail that
\[
|B_{(\lambda, 0), (\gamma, 1)}| \leq \Theta'(\lambda^{-1} \gamma) 
= \Theta'(\lambda^{-1} \gamma) \eta(1-0) = \widetilde{\Theta}((\lambda, 0)^{-1} (\gamma, 1))
\]
and
\[
|B_{(\gamma, 1), (\gamma', 1)} |\leq \Theta'(\gamma^{-1} \gamma') 
= \Theta'(\gamma^{-1} \gamma')  \eta(1-1) = \widetilde{\Theta} ((\gamma, 1)^{-1} (\gamma', 1))
\]
respectively. Thus, if \eqref{eq:wiener-type_somep} holds, then, for all $c \in \ell^p (\Gamma)$, we have
\[
\| \widetilde{A} c \|_{\ell^p \oplus \ell^p} 
= \| AP c \|_{\ell^p} + \| (I-P) c \|_{\ell^p} 
\geq C \| Pc\|_{\ell^p} + \|(I-P)c \|_{\ell^p} 
\gtrsim \| c \|_{\ell^p},
\]
where the last step follows from the equivalence $\| c \|_{\ell^p} \asymp \| P c \|_{\ell^p} + \|(I-P) c \|_{\ell^p}$. 
Thus $B$ is $p$-bounded from below. 
By Proposition \ref{prop:sjostrand-wienertype}, the matrices
 $B$ and  $\widetilde{A}$ are $q$-bounded from below on 
$\ell^2 (\Gamma^*)$ and $\ell^q (\Gamma)$, respectively, for all $q \in [1,\infty]$. 
Since 
\[ 
\| \widetilde{A} c \|_{\ell^q \oplus \ell^q} = \| A P c \|_{\ell^q} + \| (I - P) c \|_{\ell^q} \gtrsim \| c \|_{\ell^q}
\]
 for $c \in \ell^q (\Gamma)$, it follows that
$
\| A P c \|_{\ell^q} = \| \widetilde{A} P c \|_{\ell^q \oplus \ell^q} \geq C' \| P c \|_{\ell^q} 
$
for all $c \in \ell^q (\Gamma)$, which completes the proof. 
\end{proof}

\subsection{Proof of Theorem \ref{thm:frame_Riesz_universal}}
We apply Theorem \ref{thm:wiener-type-extension}. 
For this, let $\pi (\Gamma) h$ and $\{\widetilde{h}_{\gamma} \}_{\gamma \in \Gamma}$ 
be canonical dual frames as guaranteed by Proposition \ref{prop:canonical_dual}. 
Define the operators 
$A := C_{g, \Lambda} C^*_{h, \Gamma} : \ell^p (\Gamma) \to \ell^p (\Lambda)$ 
and $P := C_{\widetilde{H}} C^*_{h, \Gamma} : \ell^p (\Gamma) \to \ell^p (\Gamma)$,
where $C_{\widetilde{H}}$ is the coefficient operator of $\{\widetilde{h}_{\gamma} \}_{\gamma \in \Gamma}$. 
By the orthogonality relations \eqref{ORs}, the matrices representing $A$ and $P$ are easily seen to satisfy 
$|A_{\lambda, \gamma} | = |\langle \pi (\gamma) h, \pi(\lambda) g \rangle | \leq d_{\pi}^{1/2} (|V_g h| \ast |V_h h|) (\lambda^{-1} \gamma)$
respectively $|P_{\gamma, \gamma'}| = | \langle \pi(\gamma') h, \widetilde{h}_{\gamma} \rangle | \leq d_{\pi}^{1/2} (\widetilde{\Theta} \ast |V_h h|)(\gamma^{-1} \gamma') $, 
where $\widetilde{\Theta} \in \WRstw(G)$ is as in \eqref{eq:canonical_dual_localized}. 
Since $|V_g h| \ast |V_h h| \in \mathcal{S} (G) \subseteq \WRstw(G)$ 
and 
\[
\widetilde{\Theta} \ast |V_h h| \in W_R(L^{\infty}, L^1_{w_{\alpha}}) (G) \ast W(L^{\infty}, L_{w_{\alpha}}^1)(G) \hookrightarrow \WRstw(G)
\] for any $\alpha \geq Q+1$, the hypotheses of Theorem \ref{thm:wiener-type-extension} 
are satisfied.

(i) Suppose $\CS$ forms a $p$-frame for $\CoLp$ for some $p \in [1,\infty]$. 
For every $f \in \CoLp$, there exists a $c \in \ran (C_{\widetilde{H}})$ such that $f = C^*_{h, \Gamma} c$. 
Therefore, for any $c \in \ell^p (\Gamma)$, we have
\begin{align*}
\| A P c\|_{\ell^p} = \| C_{g, \Lambda} f \|_{\ell^p} \asymp \| f \|_{\Co (L^p(G))} 
\asymp \| C_{\widetilde{H}} f \|_{\ell^p} = \|C_{\widetilde{H}} C^*_{h, \Gamma} c \|_{\ell^p}
= \| P c \|_{\ell^p}.
\end{align*}
An application of Theorem \ref{thm:wiener-type-extension} therefore gives 
$\| A P c \|_{\ell^q} \gtrsim \| P c \|_{\ell^q}$ for all $q \in [1,\infty]$. 

To show that the system $\CS$ forms a $q$-frame for $\Co(L^q(G))$, let $f \in \Co(L^q(G))$ 
and $c := C_{\widetilde{H}} f \in \ran(C_{\widetilde{H}})$ such that $f = C^*_{h, \Gamma} c$. 
Then 
\[
\| C_{g, \Lambda} f \|_{\ell^q} 
= \| A P c \|_{\ell^q} \gtrsim \| P c \|_{\ell^q} 
= \|C_{\widetilde{H}} f \|_{\ell^q} \asymp \| f \|_{\Co(L^q(G))},
\]  
which proves (i). 

(ii) Suppose that $\| C^*_{g, \Lambda} c \|_{\CoLp} \asymp \| c \|_{\ell^p}$ 
for all $c \in \ell^p (\Lambda)$. Then $A^* = C_{h, \Gamma} C^*_{g, \Lambda}$ 
is bounded from below on all of $\ell^p (\Lambda)$. 
By Theorem \ref{thm:wiener-type-extension}, it follows then that 
$A^* = C_{h, \Gamma} C^*_{g, \Lambda}$ is bounded from below 
on all of $\ell^q (\Lambda)$ for any $q \in [1,\infty]$, 
and hence so is $C_{g, \Lambda}^* : \ell^q (\Lambda) \to \Co (L^q(G))$. 
\qed

\end{document}